%% file: main.tex
\documentclass[12pt]{article}
\usepackage{enumitem}
\usepackage[a4paper,margin=3cm,innermargin=3cm]{geometry}
\usepackage[latin9]{inputenc}
\usepackage{amsmath}
\usepackage{amssymb}
\usepackage{amstext}
\usepackage{amsfonts}
\usepackage{amsthm}
\usepackage{mathtools}
\usepackage{scrextend}
\usepackage{array}
\usepackage{multicol}
\usepackage{mathrsfs}
\usepackage{color}
\usepackage{hyperref}
\usepackage{bm}
\usepackage{fancyhdr}
\usepackage{svg}
\usepackage{float}
\usepackage[bottom]{footmisc}
\usepackage{csvsimple}
\usepackage{siunitx}
\usepackage{svg}
\usepackage{comment}

\usepackage{xurl}

\usepackage{pgf,tikz,pgfplots}
\usetikzlibrary{decorations.markings}
\pgfplotsset{compat=1.15}
\usepackage{mathrsfs}
\usetikzlibrary{arrows}
\usepackage[numbers]{natbib}
\setcitestyle{number,open={[},close={]}}

\newcommand{\R}{\mathbb{R}}

\newcounter{number}
{\begin{list}%
{{\bf ~Proposition~\arabic{number}.}}{
\usecounter{number}
\setlength{\labelwidth}{0in}
\setlength{\leftmargin}{0in}
\setlength{\rightmargin}{0in}
\setlength{\itemsep}{1em}}}%
{\end{list}}

\newcommand{\solution}[1]{}

\newcommand{\reals}{{\mathbb R}}

\newcommand{\norm}[1]{ \left\|#1\right\|}

\newcommand{\eunorm}[1]{\norm{#1}_2}

\newcommand{\fnorm}[1]{\norm{#1}_F}
\newcommand{\fnormbig}[1]{\big\| #1 \big\|_F}

\newcommand{\abs}[1]{\left| #1 \right|}
\newcommand{\inner}[2]{\Big\langle #1, #2 \Big\rangle}
\newcommand{\sphere}[1]{\mathbb{S}^{#1-1}}

\usepackage[algo2e]{algorithm2e}
\usepackage{algorithmic}
\usepackage{algorithm, caption}

\DeclareMathOperator*{\argmax}{arg\,max}
\DeclareMathOperator*{\argmin}{arg\,min}

\DeclareMathOperator{\diag}{diag}

\DeclareMathOperator{\offdiag}{offdiag}

\DeclareMathOperator{\Span}{span}

\DeclareMathOperator{\kernel}{ker}

\DeclareMathOperator{\Cohe}{Coherence}

\newcommand{\ind}{\mathsf{Ind}}

\newcommand{\tensor}{\otimes}

\newcommand{\compose}{\circ}

\newcommand{\pesudoinverse}{\ssymbol{2}}

\def\@fnsymbol#1{\ensuremath{\ifcase#1\or *\or \dagger\or \ddagger\or
   \mathsection\or \mathparagraph\or \|\or **\or \dagger\dagger
   \or \ddagger\ddagger \else\@ctrerr\fi}}
\newcommand{\ssymbol}[1]{^{\@fnsymbol{#1}}}

\newcommand{\Normal}[2]{\mathcal{N}(#1, #2)}

  \newcommand{\Acal}{{\mathcal A}}
  
  \newcommand{\Ocal}{{\mathcal O}}

   \newcommand{\Prob}{\mathrm{Prob}}   

\DeclareMathOperator*{\Exp}{{\mathbb E}}    
\newcommand{\suppress}[1]{}

\newcommand{\gap}{\mathsf{gap}}

\newtheorem{theorem}{Theorem}

\newtheorem{corollary}[theorem]{Corollary}
\newtheorem{lemma}[theorem]{Lemma}
\newtheorem{remark}[theorem]{Remark}
\newtheorem*{remark*}{Remark}

\numberwithin{equation}{section}

\DeclareMathOperator{\cum}{Cum}

\title{Randomized Joint Diagonalization \\ of Symmetric Matrices }

\author{Haoze He\footnotemark[1] \and Daniel Kressner\footnote{The work of both authors was supported by the SNSF research project \emph{Probabilistic methods for joint and
singular eigenvalue problems}, grant number: 200021L\_192049. \'Ecole Polytechnique F\'ed\'erale de Lausanne (EPFL), Institute of Mathematics, 1015 Lausanne, Switzerland. E-mails: \href{mailto:haoze.he@epfl.ch}{haoze.he@epfl.ch},  \href{mailto:daniel.kressner@epfl.ch}{daniel.kressner@epfl.ch}}}

\begin{document}
\maketitle

\begin{abstract}
Given a family of nearly commuting symmetric matrices, we consider the task of computing 
an orthogonal matrix that nearly diagonalizes every matrix in the family.
In this paper, we propose and analyze randomized joint diagonalization (RJD) for performing this task. RJD applies a standard eigenvalue solver to random linear combinations of the matrices. Unlike existing optimization-based methods, RJD is simple to implement and leverages existing high-quality linear algebra software packages. Our main novel contribution is to prove robust recovery: Given a family that is $\epsilon$-near to a commuting family, RJD jointly diagonalizes this family, with high probability, up to an error of norm $\mathcal O(\epsilon)$. We also discuss how the algorithm can be further improved by deflation techniques and demonstrate its state-of-the-art performance by numerical experiments with synthetic and real-world data.
\end{abstract}

\pagestyle{myheadings}
\thispagestyle{plain}

\input{introduction}

\input{almost_commuting}

\input{eigenvalue_vector}

\input{rjd_algorithm}

\input{perturbation_analysis}

\input{deflation_algorithm}

\input{numerical_experiments}

\input{conclusion}

\bibliography{mybib.bib}
\bibliographystyle{plainnat}
\end{document}

%% file: introduction.tex
\section{Introduction}


It is well known that a commuting family of real symmetric matrices $\Acal = \{A_k \in \reals^{n \times n}\}_{k=1}^{d}$ can be jointly diagonalized, that is, there is an orthogonal matrix $Q \in \reals^{n \times n}$ such that each matrix $Q^T A_k Q$ is diagonal. Such joint diagonalization (JD) problems arise in a variety of applications. For example, in linear instantaneous Blind Source Separation, source signals are reconstructed from the observed mixed signals by performing joint diagonalization on fourth-order cumulant matrices \cite{Cardoso93blindbeamforming}, covariance matrices of different signal segments \cite{942614} or autocorrelation
matrices \cite{554307}; see also~\cite{MiettinenJari2017Bssb} for an overview. In Latent Variable Models (LVM), parameters are estimated through orthogonal decompositions of low-order moment tensors \cite{MR3270750}, which can be further reduced to JD of slices (or linear combinations thereof) of the moment tensors \cite{MR3270750, MR2262974, pmlr-v38-kuleshov15}. JD also appears in manifold learning~\cite{EynardDavide2015MMAb}, parameter identification problems~\cite{MR3936543} and computer graphics~\cite{KovnatskyA2013Cqb}.

In the applications mentioned above, the commutativity assumption is idealistic and rarely satisfied in practice, due to noise, estimation error or even round-off error. Instead, one considers
a family of \emph{nearly} commuting matrices $\tilde{\Acal} = \{ A_k + E_k\}_{k=1}^{d}$,
where $ \Acal = \{ A_k\}_{k=1}^{d}$ is commuting and each $E_k$ is a symmetric matrix of small norm.
JD now aims at finding an orthogonal matrix $\tilde{Q}$ that \emph{nearly} diagonalizes each $\tilde{A}_k \in \tilde{\Acal}$.
Therefore, most existing methods view JD as an optimization problem that aims at minimizing
the error contained in the off-diagonal parts of $\tilde{Q}^T \tilde{A}_k \tilde{Q}$.
Specifically, one considers
\begin{equation}
\label{eq:optimization_task}
    \min_{\tilde{Q} \in \reals^{n \times n}, \tilde{Q}^T \tilde{Q} = I} \mathcal{L}(\tilde{Q})
\end{equation}
given a suitable measure $\mathcal{L}(\tilde{Q})$ for off-diagonality. A natural choice is
\begin{equation}
\label{eq:least_square_measure}
    \mathcal{L}(\tilde{Q}) := \sum_{k = 1}^{d}\big\| \offdiag(\tilde{Q}^T \tilde{A}_k \tilde{Q}) \big\|_F^2,
\end{equation}
where $\offdiag$ sets diagonal entries to zero and preserves off-diagonal entries, and $\|\cdot\|_F$ denotes the Frobenius norm. When each $\tilde{A}_k$ is symmetric positive definite, another popular choice is 
\begin{equation}
\label{eq:kl_measure}
\mathcal{L}(\tilde{Q}) := \frac{1}{2n}\sum_{k = 1}^{d}[\log \det\diag(\tilde{Q}^T\tilde{A}_k\tilde{Q}) - \log\det(\tilde{Q}^T\tilde{A}_k\tilde{Q})],
\end{equation}
with $\diag(A) = A - \offdiag(A)$. This measure was introduced by Pham in \cite{Pham2001} and can be interpreted as  the Kullback-Leibler divergence between two multivariate Gaussian distributions with covariance matrices $\tilde{Q}^T\tilde{A}_k\tilde{Q}$ and $\diag(\tilde{Q}^T\tilde{A}_k\tilde{Q})$, respectively~\cite{MR3797813}. 

Most existing optimization algorithms for solving~\eqref{eq:optimization_task} belong to one of the following three categories:
 \begin{itemize}
     \item \emph{Jacobi-like algorithms} are a special case of coordinate descent methods, in which the matrix $\tilde{Q}$ is decomposed into $n(n-1)/2$ Givens rotations (that is, rotations restricted to the plane determined by a chosen pair of coordinates) and (\ref{eq:least_square_measure}) is minimized successively for each rotation in a specified order. The first Jacobi-like algorithm for JD was introduced in \cite{MR1238912}, with the closed-form solution for the optimal angle of each rotation  derived in \cite{MR1372928};
     this algorithm uses cyclic order. In~\cite{IshtevaAbsilVan2013}, the order is chosen to ensure a sufficiently large directional derivative; global convergence to stationary points of this algorithm is proved in~\cite{LiUsevichComon2018}. In \cite{Pham2001}, a potentially non-orthogonal $\tilde{Q}$ is decomposed into $n(n-1)/2$ invertible elementary transformations and (\ref{eq:kl_measure}) is minimized successively for each elementary transform.
     
     \item 
      \emph{Quasi-Newton methods}, that is the Newton method with some approximation of the Hessian, have been proposed in~\cite{MR2247999} for minimizing~\eqref{eq:least_square_measure}, with multiplicative updates to ensure the orthogonality or invertibility of $\tilde{Q}$. In \cite{940221}, the problem of minimizing a modified variant of~\eqref{eq:least_square_measure} is reformulated as a subspace fitting problem and solved with a quasi-Newton method. In \cite{10.1007/978-3-642-00599-2_38}, a quasi-Newton method with an intrinsic scale constraint is applied to find a non-orthogonal joint diagonalizer based on the measure (\ref{eq:least_square_measure}). More recently, quasi-Newton methods based on the measure (\ref{eq:kl_measure}) have been developed in \cite{ablin:hal-01936887, de2021joint}.
      
     \item \emph{Riemannian optimization} leverages the fact that $\tilde Q$ belongs to a matrix manifold, e.g., the manifold of orthogonal matrices. In \cite{10.1007/978-3-540-30110-3_56}, Riemannian gradient descent is used for minimizing~\eqref{eq:least_square_measure}.  A trust-region-based algorithm on the oblique manifold to obtain a non-orthogonal joint diagonalizer is proposed in \cite{1661433}. In~\cite{MR3652848}, a Riemannian Newton method is applied to minimize a variant of (\ref{eq:kl_measure}). Recently, existing Riemannian optimization methods for (\ref{eq:least_square_measure}) and (\ref{eq:kl_measure}) as well as their variants have been unified into a common framework in \cite{MR4052642}. 
 \end{itemize}
 Besides these three main categories, JD optimization problems have also been approached through gradient-based methods~\cite{7080517, 10.1007/978-3-540-30110-3_12} and quadratic optimization~\cite{1677895}.
 
In principle, optimization-based JD algorithms come with major advantages: They benefit from existing optimization techniques and easily extend to more general settings by modifying the constraint~\cite{1661433, 10.1007/978-3-642-00599-2_38} or changing the parameterization~\cite{4799135, MR2247999}. However, to the best of our knowledge, none of the existing convergence results establishes global convergence to a global minimizer. Thus, none of these algorithms is guaranteed to recover a nearly diagonalizing transformation 
for a nearly commuting family. More importantly, the use of general optimization techniques often leads to computationally demanding algorithms. In particular, for $d = 1$ optimization can be expected to perform significantly worse compared to applying a standard eigenvalue solver~\cite{golubvanloan13}, such as MRRR \cite{MRRR05}, to $A_1$.

\emph{Randomized algorithms} approach JD problems in a radically different way: The diagonalizer $\tilde Q$ is extracted from applying a standard eigenvalue solver to one or several random linear combinations of the nearly commuting matrices $\tilde A_k$. Such randomized algorithms have already been successfully applied in various fields, including learning latent variables, parameter identification, and polynomial root finding \cite{ pmlr-v23-anandkumar12, MR3332930, Corless1997, MR3936543}; see Section~\ref{sec:rjd} for more details. Nevertheless, we are not aware of an analysis that would justify the use of such randomized techniques without further, possibly strong assumptions on the data. In particular, the analysis of \cite{MR3936543} requires sufficiently large gaps for the eigenvalues of the underlying ground truth matrices.

In this paper, we propose and analyze a randomized joint diagonalization (RJD) algorithm that can be significantly faster than optimization-based algorithms, while achieving a similar level of accuracy. Almost entirely based on existing standard eigenvalue solvers, RJD is very simple to implement. Moreover, we establish robust recovery for RJD: If the input error (that is, the matrices $E_k$ perturbing the commuting matrices $A_k$) has norm $\epsilon$ then the output error of RJD (that is, the square root of $\mathcal L(\tilde Q)$ from~\eqref{eq:least_square_measure}) is $\mathcal O(\epsilon)$, with high probability. Our main result (Theorem~\ref{theorem no gap}) does not require any assumption on $A_k$, in particular it imposes no assumption on eigenvalue gaps.

The rest of this paper is organized as follows. In Section \ref{sec:rjd}, the basic RJD is introduced and it is shown to exactly recover a joint diagonalizer $Q$ for a commuting family. Section~\ref{sec:perturbationanalysis} extends this result to robust recovery. Improvements of the basic algorithm with deflation techniques are discussed in Section \ref{sec:deflation}. In Section \ref{sec:numeric}, we demonstrate the accuracy and efficiency of our algorithms through various numerical experiments, including synthetic data, Blind Source Separation and Single Topic Models.  

Independently of this work, Sutton \cite{Sutton23} recently developed a new, deterministic method for jointly diagonalizing $d=2$ symmetric matrices and established a robust recovery guarantee for that method.

%% file: almost_commuting.tex
\subsection{Nearly commuting vs almost commuting}

As discussed above, we consider a \emph{nearly} commuting family of matrices in this work, that is, a (small) perturbation makes the matrices commute. In the literature, one can also find the notion of \emph{almost} commuting matrices. In particular, a pair of matrices $\tilde A, \tilde B$ is called \emph{almost} commuting if the commutator $[\tilde A, \tilde B] = \tilde A \tilde B - \tilde B \tilde A$ has small norm.

By the triangle inequality, \emph{nearly} commuting implies \emph{almost} commuting. However, the converse direction is much more subtle. Given $n\times n$ symmetric $\tilde A, \tilde B$ with $\|[\tilde A, \tilde B]\|_2 \leq \delta$, where $\| \cdot \|_2$ denotes the spectral norm, one needs to determine symmetric perturbations $E_A, E_B$, of spectral norm $\epsilon(\delta)$ not much larger than $\delta$, such that $\tilde A + E_A$ and $\tilde B + E_B$ commute. This problem has been studied for decades in mathematical physics and operator theory. In \cite{lin1997almost}, it was  proven that there exist perturbations such that $\epsilon(\delta)$ converges to zero as $\delta$ converges to zero and $\epsilon(\delta)$ does not depend on other properties of $\tilde A, \tilde B$ (such as $n$). 
This qualitative result was improved to $\epsilon(\delta) =\Ocal({\delta ^{1/5}})$ in~\cite{hastings2009making} and, more recently, to $\epsilon(\delta) = \Ocal(\delta ^{1/2})$ in \cite{kachkovskiy2016distance}, which is optimal.

It is simple to compute the norm of commutators, but the discussion above suggests that this might not be the most appropriate measure when commutativity is violated due to errors in the input data, such as roundoff error or noise.
Let us remark that the work of this paper allows one to verify, with high probability, whether a matrix family is nearly commuting by attempting to jointly diagonalize the family with our algorithms.

%% file: eigenvalue_vector.tex
\subsection{Common eigenvectors and eigenvalue vectors}
\label{subsec:commoneigenvector}

This section summarizes notation used throughout this work. Given a commuting family of symmetric matrices $\mathcal A = \{A_k\}_{k=1}^{d}$, we call $x \not = 0$ a \emph{common eigenvector} of $\mathcal{A}$ if $x$ is an eigenvector of each $A_k \in \mathcal{A}$, that is, $A_k x = \lambda^{(k)} x$ for some (eigenvalue) $\lambda^{(k)} \in \mathbb R$. The vector $\Lambda = [\lambda^{(1)},\ldots,\lambda^{(d)}]^T$ collecting these eigenvalues is called \emph{eigenvalue vector}.
In total, there are $n$ (counting multiplicities) such eigenvalue vectors
\[
\Lambda_i = [\lambda^{(1)}_i,\ldots,\lambda^{(d)}_i]^T, \quad i = 1,\dots, n.
\]
A subspace $\mathcal X \subset \reals^n$ is called a \emph{common invariant subspace} of $\Acal$ if it is spanned by common eigenvectors.


%% file: rjd_algorithm.tex
\section{Basic RJD algorithm}
\label{sec:rjd}

\subsection{The algorithm}

The basic idea of randomized joint diagonalization is to reduce a joint eigenvalue problem to one or several standard eigenvalue problems through random linear combinations. More specifically, consider a standard Gaussian random vector $\mu\sim \Normal{0}{I_d}$, that is, the entries $\mu_k$, $k = 1,\ldots, d$, are i.i.d. standard normal random variables. Given a (nearly commuting) family of real symmetric matrices $\tilde{\Acal} = \{\tilde{A}_k\}_{k=1}^{d}$, we then define 
\[\tilde{A}(\mu) = \sum_{k=1}^{d} \mu_k \tilde{A}_k. \]
Because $\tilde{A}(\mu)$ is symmetric, there exists an orthogonal matrix $\tilde{Q} \in \reals^{n \times n} $ such that $\tilde{Q}^T \tilde{A}(\mu) \tilde{Q}$ is diagonal. As we will prove, this matrix is often a good approximate joint diagonalizer and the success probability can be boosted by multiple independent trials. Algorithm~\ref{alg:rjd}, called RJD in the following, summarizes these ideas.

Approaches similar to Algorithm~\ref{alg:rjd} have appeared several times in the literature and in different fields. In learning latent variable models, a joint diagonalizer of dimension reduced and column scaled slices of the supersymmetric third-order moment tensor is obtained through diagonalizing one random linear combination in \cite[Algorithm A]{pmlr-v23-anandkumar12}. Note that the involved matrices are not necessarily symmetric {because of the scaling}. In \cite[Algorithm 1]{MR3332930}, the whitened slices of the third order moment tensor, which are symmetric, are jointly diagonalized through diagonalizing one random linear combination. In this case, involved matrices are symmetric and this approach is equivalent to RJD with one trial. In parameter identification problems,  a family of nearly commuting matrices that are potentially non-normal are jointly diagonalized through diagonalizing complex random linear combinations in \cite[Algorithm 1]{MR3936543}. In polynomial root finding, a generic convex combination of nearly commuting matrices is used in \cite{Corless1997} to bring each matrix in the family into nearly block upper triangular form.

\begin{algorithm}[H]
\caption{\textbf{R}andomized \textbf{J}oint \textbf{D}iagonalization (RJD)}
\textbf{Input:} \text{Family of symmetric matrices $\tilde{\Acal} =\{\tilde{A}_k \in \reals^{n \times n}\}_{k=1}^{d}$, number of trials $L$.}\\
 \textbf{Output:} \text{Joint approximate diagonalizer $\tilde{Q}$.}
\begin{algorithmic}
\label{alg:rjd}
\FOR{$i=1$ to $L$}
    \STATE Draw $\mu^{(i)}$ from distribution $\Normal{0}{I_d}$.
    \STATE Compute $\tilde{A}(\mu^{(i)}) = \mu_1^{(i)} \tilde A_1 + \cdots + \mu_d^{(i)} \tilde A_d$.
    \STATE
    Compute $\tilde{Q}_i$ that diagonalizes $\tilde{A}(\mu^{(i)})$.
\ENDFOR
\STATE Let $i^* = \argmin_{i}\Big\{ \sum_{k = 1}^{d}\fnormbig{\offdiag(\tilde{Q}_i^T \tilde{A}_k \tilde{Q}_i)}^2\Big\}$.
\RETURN $\tilde Q \equiv \tilde{Q}_{i^*}$
\end{algorithmic}
\end{algorithm}
Because the Frobenius norm satisfies \[\|\tilde A_k\|_F^2 = \|\diag(\tilde{Q}_i^{T} \tilde{A}_k \tilde{Q}_i)\|_F^2 + \|\offdiag(\tilde{Q}^T_i \tilde{A}_k \tilde{Q}_i)\|_F^2,\] it suffices to compute the diagonal elements and select the best trial in Algorithm~\ref{alg:rjd} through
\[i^* = \argmax_{i}\Big\{ \sum_{k = 1}^{d}\fnormbig{\diag(\tilde{Q}_i^T \tilde{A}_k \tilde{Q}_i)}^2\Big\}.\]

\subsection{Correctness for exactly commuting matrices}\label{subsec:exact_correct}
Applied to a commuting family of symmetric matrices $\mathcal A = \{A_k\}_{k = 1}^d$, Algorithm \ref{alg:rjd} returns an exact joint diagonalizer with probability $1$. To see this, we first note that every eigenvalue $\lambda(\mu)$ of $A(\mu)$ can be expressed as the inner product with an eigenvalue vector 
$\Lambda = [\lambda^{(1)},\ldots,\lambda^{(d)}]^T$ of $\mathcal A$:
\begin{equation} \label{eq:eiginnerproduct}
\lambda(\mu) = \mu_1 \lambda^{(1)} + \cdots + \mu_d \lambda^{(d)} = \langle \mu, \Lambda \rangle.
\end{equation}
\begin{lemma} \label{lemma:multieigenvalues}
With the notation introduced above, let $\lambda(\mu) = \langle \mu, \Lambda \rangle$ and 
$\tilde \lambda(\mu) = \langle \mu, \tilde \Lambda \rangle$ be eigenvalues of $A(\mu)$ for $\mu \sim \Normal{0}{I_d}$. Then 
$\lambda(\mu) = \tilde \lambda(\mu)$ implies $\Lambda = \tilde \Lambda$ with probability 1.
\end{lemma}
\begin{proof}
Because of $\lambda(\mu) - \tilde \lambda(\mu) = \langle \mu,\Lambda-\tilde \Lambda \rangle \sim \Normal{0}{\|\Lambda-\tilde \Lambda\|_2^2}$, with $\|\cdot\|_2$ denoting the Euclidean norm, it follows that $\lambda(\mu) - \tilde \lambda(\mu) = 0$ happens with probability $0$ when $\Lambda \not= \tilde \Lambda$.
\end{proof}

\begin{theorem}
\label{thm:exactlycommuting}
Let $Q$ be an orthogonal matrix that diagonalizes $A(\mu) = \mu_1 A_1 + \cdots + \mu_d A_d$ for a commuting family of symmetric matrices $\mathcal A = \{A_k\}_{k = 1}^d$. Then 
$Q$ also diagonalizes each matrix $A_1,\ldots,A_d$ with probability $1$. 
\end{theorem}
\begin{proof}
The result follows from standard linear algebra arguments \cite{Strang1988}; we include the proof for the sake of completeness.
Without loss of generality, we may assume that the first $n_1 \ge 1$ columns of $Q$ span the eigenspace 
$\mathcal X_1$ belonging to an eigenvalue $\lambda_1(\mu)$ of $A(\mu)$. Then
\[
Q^T A(\mu) Q = \begin{bmatrix}
 \lambda_1(\mu) I & 0 \\
 0 & A_{22}(\mu)
 \end{bmatrix}, \quad Q^T A_k Q
= \begin{bmatrix}
 A_{11}^{(k)} & \big( A_{21}^{(k)} \big)^T \\
 A_{21}^{(k)} & A_{22}^{(k)}
 \end{bmatrix}.
\]
Because $A(\mu)$ and $A_k$ commute, one has $(A_{22}(\mu) - \lambda_1(\mu) I)A_{21}^{(k)} = 0$. Because   $\lambda_1(\mu)$ is not an eigenvalue of $A_{22}(\mu)$, this implies $A_{21}^{(k)} = 0$. By  Lemma~\ref{lemma:multieigenvalues}, each $A_{11}^{(k)}$ has $n_1$ equal eigenvalues with probability $1$ and is, in turn, a multiple of the identity matrix. This completes the proof by induction.
\end{proof}

As a curiosity we note that existing results on exact recovery of $Q$ for commuting matrices often assume  distinct eigenvalues. For example, in \cite{MR3332930} exact recovery is established through showing that $A(\mu)$ has distinct eigenvalues with probability $1$ for the application under consideration.
 In \cite[Corollary 2.2]{MR3936543}, it is shown that a diagonalizer of $A(\mu)$ diagonalizes a family of commuting, possibly nonnormal matrices if $A(\mu)$ has distinct eigenvalues.

%% file: perturbation_analysis.tex
\section{Analysis of RJD}
\label{sec:perturbationanalysis}

In this section, we extend the statement of Theorem~\ref{thm:exactlycommuting} from exact to robust recovery.
Instead of an exactly commuting $\Acal = \{A_k\}_{k=1}^d$, we consider a nearly commuting family \begin{equation}
    \label{eq:nearly_commuting_family}
    \tilde{\Acal}~=~\{\tilde{A}_k = A_k + E_k\}_{k=1}^d,
\end{equation} where each $E_k$ is symmetric and $\fnorm{E_k}$ is assumed to be small. Robust recovery means that Algorithm $\ref{alg:rjd}$ returns $\tilde Q$ such that each $\tilde Q^T \tilde A_k \tilde Q$ has small off-diagonal error with high probability. We will establish two types of results: Theorem~\ref{nonasym theorem}
establishes an error bound that depends on the smallest gap between the eigenvalue vectors of the (unknown) groundtruth $\Acal$. Theorem~\ref{theorem no gap} removes this gap dependence, at the cost of a more pronounced dependence on $n$. To the best of our knowledge this is the first robust recovery result for any algorithm that aims at diagonalizing a nearly commuting family.

\subsection{Preliminary results}

\label{subsec:preliminary}
Our analysis will be based on the following results from probability and matrix perturbation theory.
\begin{lemma}\label{probability Bound}
Let $v  \in \reals^d$ with $\eunorm{v} = 1$. Suppose $d \geq 2$. If $u \in \reals^d$ follows a uniform distribution over the unit sphere $\sphere{d}$,  then for $C > 0$, the probability that
\[\abs{\langle u, v \rangle} \leq C \]
holds is at most $\sqrt{\frac{2d}{\pi}} C$.
\end{lemma}
\begin{proof}
See \cite[Lemma and Section 3]{MR708458}. See also \cite[Section 4.1.1]{lotz2020wilkinsons} and \cite[Lemma 5.2]{banks2022global}.
\end{proof}

\begin{lemma}
\label{lemma:invsubspace}
Let $A,E \in \reals^{n \times n}$ be symmetric with $\fnorm{E} = \epsilon$. Consider an invariant subspace $\mathcal{X}_0$ of $A$ with orthonormal basis $X_0 \in \reals^{n \times k}$ and let $X_0^\perp$ be any matrix such that $\begin{bmatrix}
 X_0 & X_0^\perp
\end{bmatrix}$ is an orthogonal matrix. Partition
\[  A = \begin{bmatrix}
 X_0 & X_0^\perp
\end{bmatrix} \begin{bmatrix}
A_{11}& 0 \\ 0 & A_{22}
\end{bmatrix} \begin{bmatrix}
 X_0 & X_0^\perp
\end{bmatrix}^T, \quad A_{11} \in \R^{k\times k}, \quad A_{22} \in \R^{(n-k)\times (n-k)},  \]
and assume
$$s_0 := \min\{\abs{\lambda -\nu}: \lambda \in \sigma(A_{11}),\nu \in \sigma(A_{22})\} > 0,\qquad \epsilon  < s_0/4,$$
where $\sigma(\cdot)$ denotes the spectrum of a matrix. Then the perturbed matrix $A+E$ has a unique invariant subspace $\mathcal{X}_E$ of dimension $k$ belonging to the eigenvalues within $\epsilon$ distance of $\sigma(A_{11})$. Moreover,
for any orthonormal basis $X_E$ of $\mathcal{X}_E$, there exists a basis $\hat X_0 $ of $\mathcal{X}_0$, $B \in \reals^{(n-k) \times k}$ and $C \in \reals^{k \times k}$ such that 
\[X_E = \hat X_0 + X_0^\perp \mathbb{T}^{-1}(B)C\]
with 
\[
    \|\hat X_0\|_2 \leq 1, \quad \fnorm{\mathbb{T}^{-1}(B)} \leq 1/2, \quad \fnorm{B} \leq \sqrt{5}\fnorm{EX_0}, \quad \eunorm{C} \leq 1,
\]
where $\|\cdot\|_2$ denotes the spectral norm and $\mathbb{T}$ is the Sylvester operator defined by
\[\mathbb{T}: \reals^{(n-k) \times k } \to  \reals^{(n-k) \times k }, \quad \mathbb{T}(X) = XA_{11} - A_{22}X.\]
\end{lemma}
\begin{proof}
The result essentially follows from existing perturbation results on invariant subspaces~\cite{StewartSun1990,MR3205738}.

It is simple to check that the result is invariant under the choice of bases $X_0, X_0^\perp$ and that we may, in fact, assume without loss of generality that $\begin{bmatrix}
 X_0 & X_0^\perp
\end{bmatrix} = I_n$ and, hence, $A = \begin{bmatrix}
A_{11}& 0 \\ 0 & A_{22}
\end{bmatrix}$. We partition $E = \begin{bmatrix}
E_{11}& E_{12} \\ E_{21} & E_{22}
\end{bmatrix}$ conformally. 

The $n\times k$ matrix $\hat{X}_E = \begin{bmatrix}
I \\ W_E
\end{bmatrix}$ spans an invariant subspace $\mathcal{X}_E$  of $A+E$ if and only if $W_E$ satisfies
\begin{equation} \label{eq:quadmateq_2}
\mathbb{T}(W_E) = 
\begin{bmatrix}
 -W_E & I
\end{bmatrix}E
\begin{bmatrix}
 I \\
 W_E
\end{bmatrix}.
\end{equation}
Theorem 3.1 in~\cite{MR3205738} and Theorem 2.7 in~\cite[Chapter V]{StewartSun1990} (or Theorem 2.5 in~\cite{MR3205738})
show that this quadratic matrix equation admits a solution $W_E$ satisfying the bounds
\begin{equation}
\label{eq:W_E_bound_2}
   \fnorm{W_E} \le \frac{2}{s_0}\fnorm{E} \leq \frac{1}{2}, \quad
\fnorm{W_E} \leq \frac{2}{s_E}\fnorm{E_{21}}, 
\end{equation}
where $s_E := s_0 - \fnorm{E_{11}} - \fnorm{E_{22}} > 2\epsilon$. Theorem 3.1 in \cite{MR3205738} also shows that the eigenvalues of $A+E$ belonging to $\mathcal{X}_E$ are within $\epsilon$ distance of $\sigma(A_{11})$. By the assumption on $\epsilon$ and the Wely perturbation Theorem \cite[Corollary 4.10, Chapter IV]{StewartSun1990} (i.e., $|\lambda_i(A+E) - \lambda_i(A)| \leq \|E\|_2 \leq \|E\|_F = \epsilon$), no other eigenvalue of $A+E$ can be that close to $\sigma(A_{11})$, and hence the columns of $
X_E$ span the unique invariant subspace having this property.
Using~\eqref{eq:quadmateq_2} and \eqref{eq:W_E_bound_2}, we thus obtain
\begin{align*}
\|\mathbb{T}(W_E)\|_F \leq& \left\| \begin{bmatrix}
 -W_E & I
\end{bmatrix} \right\|_2 \cdot \left\| E
\begin{bmatrix}
 I \\
 W_E
\end{bmatrix}\right\|_F \le \frac{\sqrt{5}}{2} \left\| E
\begin{bmatrix}
 I \\
 W_E
\end{bmatrix}\right\|_F \\ \le&  \frac{\sqrt{5}}{2} \left( \|E X_0\|_F + \left\| E
\begin{bmatrix}
 0 \\
 W_E
\end{bmatrix}\right\|_F  \right) 
\le \frac{\sqrt{5}}{2} \left( \|E X_0\|_F + \frac{2\epsilon }{s_E} \|E_{21}\|_F \right) \\
\le & \sqrt{5} \|E X_0\|_F.
\end{align*}
 We obtain a (particular) orthonormal basis $\tilde X_E$ of $\mathcal X_E$ by setting
 $C := ( I + W_E^T W_E)^{-1/2}$ and
\[
\tilde{X}_E := \hat X_E  C = \hat X_0 + \begin{bmatrix}
 0 \\ \mathbb{T}^{-1} (B) 
\end{bmatrix} C,
\]
where
$\hat X_0 = X_0 C$ and $B = 
\mathbb{T}(W_E)$.  Because of $I \succeq ( I + W_E^T W_E)^{-1/2}$, it follows from matrix monotonicity that $\|\hat X_0\|_2 \leq 1$, and 
$\fnorm{B} = \|\mathbb{T}(W_E)\|_F \le \sqrt{5} \|E X_0\|_F$. By \eqref{eq:W_E_bound_2}, $\|\mathbb{T}^{-1}(B)\|_F = \|W_E\|_F \leq 1 / 2$. This completes the proof, using that any other orthonormal basis $X_E$ of $\mathcal X_E$ is related to $\tilde X_E$ via an orthonormal change of basis.
\end{proof}
When $\mathcal{X}_0$ is an eigenspace (that is, $\sigma(A_{11})$ contains only a single eigenvalue) then the statement of Lemma~\ref{lemma:invsubspace} simplifies.
\begin{corollary}
\label{lemma:eigenspace}
Under the setting and assumptions of Lemma~\ref{lemma:invsubspace}, assume additionally  that $\sigma(A_{11}) = \{\lambda_0\}$. Then the conclusions of the lemma hold with
\[
    X_E = \hat X_0  + (\lambda_0 I - A)\pesudoinverse B, \quad \|\hat X_0\|_2 \le 1, \quad \fnorm{B} \leq \sqrt{5}\fnorm{EX_0},
\]
where $\pesudoinverse$ denotes the  Moore-Penrose inverse.
\end{corollary}
\begin{proof}
The result of Lemma~\ref{lemma:invsubspace} yields
\[X_E = \hat X_0 + X_0^\perp \mathbb{T}^{-1}(B')C\]
with $\|\hat X_0\|_2 \le 1$, $\fnorm{B'} \leq \sqrt{5}\fnorm{EX_0}$ for some $B' \in \reals^{(n-k) \times k}$, and $\eunorm{C} \leq 1$. Taking into account that $A_{11} = \lambda_0 I_k$ and using that $X_0^\perp$ is an orthonormal basis of $\mathcal{X}_0^\perp$, we obtain for the second term that
\[ 
X_0^\perp \mathbb{T}^{-1}(B')C = X_0^\perp ( \lambda_0 I_{n-k} - A_{22} )^{-1} B' C = 
( \lambda_0 I_n - A )^\dagger  X_0^\perp B' C.
\]
Setting $B = X_0^\perp B'C$ concludes the proof. 
\end{proof}

\subsection{Probabilistic bound with gap}
\label{subsec:withgap}

The following theorem establishes a first probabilistic error bound, which depends on the gap between eigenvalue vectors. 
\begin{theorem}
\label{nonasym theorem}
Given a family of commuting symmetric matrices $\Acal = \{A_k \in \R^{n\times n}\}_{k=1}^{d}$,
let $m$ be the number of mutually distinct eigenvalue vectors $\Lambda_1,\ldots,\Lambda_m \in \R^d$ and set $\gap := \min_{i \neq j}\| \Lambda_i - \Lambda_j\|_2>0$. 

Let $\tilde Q$ denote the output of Algorithm \ref{alg:rjd} with $L = 1$ trial applied to $\tilde \Acal = \{\tilde{A_k} = {A_k} + E_k\}$ for symmetric $E_k$ satisfying $(\fnorm{E_1}^2 + \cdots + \fnorm{E_d}^2)^{1/2} \leq \epsilon$. 
Then for any $\epsilon > 0$ and $R > 1$, it holds that

\[ \Prob\Big( \Big( \sum_{k=1}^{d}\fnormbig{\offdiag(\tilde{Q}^T\tilde{A}_k\tilde{Q})}^2 \Big)^{1/2} \leq R \epsilon  \Big)\geq 
1 - \sqrt{\frac{d}{2\pi}} m(m-1)\gamma,\]
where $\gamma = \max\{\sqrt{5d}/ (R-1), 4\epsilon/\gap \}$.
\end{theorem}

\begin{proof}
By Lemma~\ref{lemma:multieigenvalues}, $A(\mu)$ has (with probability one) $m$ mutually distinct eigenvalues $\lambda_i(\mu) = \langle \mu, \Lambda_i \rangle$. Let $n_i$ denote the multiplicity of for $\lambda_i(\mu)$. 
We now assume that 
\begin{equation}\label{eq:eigengap_assum}
    \fnorm{E(\mu)} < |\langle\Lambda_i - \Lambda_j,\mu\rangle | / 4, \quad \forall i \neq j,
\end{equation}
where $E(\mu) := \mu_1E_1 + \cdots + \mu_d E_d$. This assumption allows us to apply Corollary~\ref{lemma:eigenspace} to each eigenspace $\mathcal X_i = \kernel(\lambda_i(\mu)I - A(\mu))$, where $\kernel$ denotes the kernel, and conclude that, after a suitable permutation of its columns, the matrix $\tilde Q$ takes the form
\[
\tilde Q = [\tilde X_1,\ldots,\tilde X_m]
\]
with
\begin{equation} \label{eq:tildexi}
\tilde{X}_i = \hat X_i + (\lambda_i(\mu)I - A(\mu))^\dagger B_i,
\end{equation}
where $\hat{X}_i$ is a basis of $\mathcal X_i$ and $\|B_i\|_F \leq \sqrt{5} \fnorm{E(\mu)X_i}$ for an orthonormal basis $X_i$ of $\mathcal{X}_i$.
We get
\begin{align}
        \big\| \offdiag(\tilde{Q}^T{A}_k\tilde{Q}) \big\|_F^2 \le & 
        \big\| \tilde{Q}^T{A}_k\tilde{Q} - \text{diag}\big(\lambda^{(k)}_1 I_{n_1}, \ldots, \lambda^{(k)}_m I_{n_m}\big)  \big\|_F^2  \nonumber \\ 
        = & \big\| {A}_k\tilde{Q} -  \tilde{Q}\, \text{diag}\big(\lambda^{(k)}_1 I_{n_1}, \ldots, \lambda^{(k)}_m I_{n_m}\big)  \big\|_F^2      \label{eq:offdiag_errornew} \\
        =& \big\| ( {A}_k - \lambda^{(k)}_1 I ) \tilde X_1\big\|_F^2 + \cdots + 
        \big\| ( {A}_k - \lambda^{(k)}_m I ) \tilde X_m \big\|_F^2. \nonumber
\end{align}
By Theorem~\ref{thm:exactlycommuting},  $\mathcal X_i$ is (with probability 1) contained in the eigenspace belonging to the eigenvalue $\lambda^{(k)}_i$ of ${A}_k$. We therefore obtain from~\eqref{eq:tildexi} that
\begin{align*}
\big\| ( {A}_k - \lambda^{(k)}_i I ) \tilde X_i\big\|_F =& 
\big\| ( {A}_k - \lambda^{(k)}_i I ) (\lambda_i(\mu)I - A(\mu))^\dagger B_i\big\|_F
\\ \le & \sqrt{5} \big\| ( {A}_k - \lambda^{(k)}_i I ) (\lambda_i(\mu)I - A(\mu))^\dagger \big\|_2\|E(\mu)X_i\|_F.
\end{align*}
Plugging this inequality into~\eqref{eq:offdiag_errornew} and setting
\begin{equation}
\label{def:cmax}
    C_{\max}(\mu) := \max_{i,k}\big\{\big\| (A_k - \lambda_{i}^{(k)}I) (\lambda_{i}(\mu)I - A(\mu))\pesudoinverse \big\|_2\big\} = \max_{k,j > i} \frac{|\lambda_i^{(k)} - \lambda_j^{(k)} |}{|\langle\Lambda_i - \Lambda_j,\mu\rangle|},
\end{equation}
we thus have
\[
\|\offdiag(\tilde{Q}^T {A}_k\tilde{Q}) \|_F \le \sqrt{5} 
C_{\max}(\mu) \sqrt{ \sum_i \|E(\mu)X_i\|_F^2} = \sqrt{5} 
C_{\max}(\mu)\|E(\mu)\|_F. \]
In turn,
\begin{align}
 \Big( \sum_k \|\offdiag(\tilde{Q}^T \tilde{A}_k\tilde{Q}) \|^2_F \Big)^{1/2}
 \le & \Big( \sum_k \|\offdiag(\tilde{Q}^T {A}_k\tilde{Q}) \|^2_F \Big)^{1/2} + \epsilon \nonumber \\
 \le & \sqrt{5d} 
C_{\max}(\mu)\|E(\mu)\|_F + \epsilon \nonumber\\
\le &
( 1 + \sqrt{5d}C_{\max}(\mu)\eunorm{\mu}) \epsilon \label{eq:bound_total_error},
\end{align}
where we used the Cauchy-Schwarz inequality $\fnorm{E(\mu)} \leq \eunorm{\mu} \epsilon$. In other words, $R\epsilon$ is an upper bound on the off-diagonal norm provided that
\begin{equation} \label{eq:boundcmax}
    C_{\max}(\mu)\eunorm{\mu} \leq ( R-1 ) / \sqrt{5d}
\end{equation}
is satisfied.

It remains to bound the probability that the inequalities~\eqref{eq:eigengap_assum} or~\eqref{eq:boundcmax} fail.
We will use union bounds and first treat~\eqref{eq:boundcmax}. For this purpose, we rewrite~\eqref{def:cmax} as
\begin{equation} \label{eq:cmaxnew}
C_{\max}(\mu) = \max_{j>i} C_{(i,j)}(\mu), \qquad 
C_{(i,j)}(\mu):= \max_k \frac{|\lambda_i^{(k)} - \lambda_j^{(k)} |}{|\langle\Lambda_i - \Lambda_j,\mu\rangle|}.
\end{equation}
For fixed $(i,j)$, choose $k^*$ to maximize the last expression. Then
\begin{align}
    \Prob\Big( \frac{R-1}{\sqrt{5d}} < C_{(i,j)}(\mu) \|\mu\|_2 \Big) 
    &= \Prob\Big( \Big|\Big\langle \Lambda_i - \Lambda_j, \frac{\mu}{\|\mu\|_2} \Big\rangle \Big|  < \frac{\sqrt{5d}}{R-1} |\lambda_i^{(k^*)} - \lambda_j^{(k^*)} |   \Big) \nonumber\\
    &\le  \Prob\Big( \Big|\Big\langle \frac{\Lambda_i - \Lambda_j}{\|\Lambda_i - \Lambda_j\|_2}, \frac{\mu}{\|\mu\|_2} \Big\rangle \Big| < \frac{\sqrt{5d}}{R-1}    \Big), \label{eq:inequ2}
\end{align}
where we used that $|\lambda_i^{(k^*)} - \lambda_j^{(k^*)} |\le \|\Lambda_i - \Lambda_j\|_2$.

Inequality~\eqref{eq:eigengap_assum} fails for fixed $(i,j)$ with probability
 \begin{align*} 
 \Prob\Big( \frac{|\langle\Lambda_i - \Lambda_j,\mu\rangle |}{4} \le \fnorm{E(\mu)} \Big)
 \le & \Prob\Big(\Big| \inner{\frac{\Lambda_i - \Lambda_j}{\eunorm{\Lambda_i - \Lambda_j}}}{\frac{\mu}{\eunorm{\mu}}} \Big| < \frac{4\epsilon}{\gap} \Big),
 \end{align*}
where we used $\fnorm{E(\mu)} \leq \eunorm{\mu}\epsilon$. Combined with~\eqref{eq:inequ2}, this establishes the following bound on the probability that~\eqref{eq:eigengap_assum} or~\eqref{eq:boundcmax} fails for fixed $(i,j)$:
 \[\Prob\Big(\Big| \inner{\frac{\Lambda_i - \Lambda_j}{\eunorm{\Lambda_i - \Lambda_j}}}{\frac{\mu}{\eunorm{\mu}}}\Big| < \gamma\Big) \le \sqrt{\frac{2d}{\pi}}\gamma,\quad \gamma = \max\Big\{\frac{\sqrt{5d}}{R-1}, \frac{4\epsilon}{\gap}\Big\}\]
where we applied Lemma~\ref{probability Bound}, noting that $\mu / \|\mu\|_2$ is distributed uniformly over the unit sphere.
Applying the union bound for the $m(m-1) / 2$ different pairs  $(i,j)$ with $j>i$ yields the bound 
\[ \sqrt{d} / \sqrt{2\pi}\cdot m(m-1)\gamma\]
on the total failure probability, which completes the proof.
\end{proof}

The result of Theorem~\ref{nonasym theorem} depends on the choice of factor $R > 1$, which controls the extent to which the input error $\epsilon$ is magnified in the output. Clearly, $R$ needs to be chosen to sufficiently large in order to get a nontrivial bound on the success probability.  When the inequality $R \leq 1+ \frac{\sqrt{5d}}{4}\frac{\gap}{\epsilon}$ holds, 
which requires $\epsilon$ to remain small relative to $\gap$,  the result of Theorem~\ref{nonasym theorem} reads as
\begin{equation} \label{eq:claim1} \Prob\Big( \Big( \sum_{k=1}^{d}\fnormbig{\offdiag(\tilde{Q}^T\tilde{A}_k\tilde{Q})}^2 \Big)^{1/2} \leq R \epsilon  \Big)\geq 
1 - \sqrt{\frac{5}{2\pi}}\frac{dm(m-1)}{(R-1)},
\end{equation}
that is, the failure probability is inversely proportional to $R-1$.
Otherwise, for $R  > 1+\frac{\sqrt{5d}}{4}\frac{\gap}{\epsilon}$, we obtain
\begin{equation*} \label{eq:claim2}
\Prob\Big( \Big( \sum_{k=1}^{d}\fnormbig{\offdiag(\tilde{Q}^T\tilde{A}_k\tilde{Q})}^2 \Big)^{1/2} \leq R \epsilon  \Big)\geq  1 -\frac{4\sqrt{d}m(m-1)\epsilon}{\sqrt{2\pi} \gap}.
\end{equation*}

\begin{remark}
\label{asymptotic}
If one uses asymptotic perturbation results, such as~\cite[Lemma 2.3]{MR3205738}, instead of Lemma~\ref{lemma:invsubspace}, the arguments from the proof of Theorem \ref{nonasym theorem} yield 
\[\Prob\Big(\sum_{k=1}^{d}\fnormbig{\offdiag(\tilde{Q}^T\tilde{A}_k\tilde{Q})}^2 \leq R^2 \epsilon^2 + \Ocal(\epsilon^3)\Big) \geq 1 - \frac{1}{\sqrt{2\pi}} \frac{dm(m-1)}{(R-1)}\]
for $R>1$. Compared to~\eqref{eq:claim1}, the constant $\sqrt{5}$ is removed. On the other hand, the constants involved in the $\Ocal(\epsilon^3)$ term critically depend on $\gap$.
\end{remark}

A first-order  analysis~\cite{Afsari07,Cardoso95perturbationof} suggests that the optimal approximate joint diagonalizer $\tilde Q$, defined as the minimizer of~\eqref{eq:least_square_measure}, becomes very sensitive to perturbations of the input data in the presence of small eigenvalue gaps. The result~\eqref{eq:claim1} implies that this increased sensitivity for small gaps does not translate into a magnification of the error. It still affects the admissible range for $R$, a shortcoming that will be removed in the following.

\subsection{Probabilistic bound without gap}
\label{subsec:withoutgap}

In this section, we state and prove our main result, a probabilistic error bound independent of gaps between eigenvalue vectors. For the sake of the analysis, we will group the eigenvalue vectors $\Lambda_1,\ldots,\Lambda_n \in \reals^d$ of $\mathcal A$ and the corresponding common eigenvectors $x_1,\ldots, x_n$ defined in Section \ref{subsec:commoneigenvector} into $m$ clusters as follows. Given $\delta > 0$, each vector $\Lambda_i$ is assigned to the cluster $K(i) \in \{1,\ldots,m\}$ such that
\begin{equation} \label{eq:clustercond}
    \|\Lambda_i - \Lambda_j\|_2 \left\{ \begin{array}{ll}
    > \delta & \text{if $K(i)\not=K(j)$}, \\
    \le \delta n_{K(i)} & \text{if $K(i)=K(j)$},
    \end{array} \right.
\end{equation}
where $n_i = \#\{j: K(j) = i\}, i = 1,\ldots,m$ denotes the cardinality of the $i$th cluster. Such a clustering can be obtained by putting $\Lambda_1$ into the first cluster and adding all vectors within $\delta$-distance to this cluster. This procedure is repeated for the remaining eigenvalue vectors to create the second cluster, etc. This  clustering is known as Single Linkage clustering with distance upper bound $\delta$ in the literature \cite[Chapter 22]{shalev-shwartz_ben-david_2014}. The corresponding common eigenvectors $x_1,\ldots,x_n$ are also grouped accordingly.

\begin{theorem}
\label{theorem no gap}
Given a family of commuting symmetric matrices $\Acal = \{A_k \in \R^{n\times n}\}_{k=1}^{d}$, let $\tilde Q$ denote the output of Algorithm \ref{alg:rjd} with $L = 1$ trial applied to $\tilde \Acal = \{\tilde{A_k} = {A_k} + E_k\}$, with symmetric $E_k$ satisfying $(\fnorm{E_1}^2 + \cdots + \fnorm{E_d}^2)^{1/2} \leq \epsilon$. 
Then for any $\epsilon > 0$ and $R > 1$, it holds that 
\[\Prob \Big(\Big(\sum_{k=1}^{d}\fnormbig{\offdiag(\tilde{Q}^T\tilde{A}_k\tilde{Q})}^2\Big)^{1/2} \leq R \epsilon \Big) \geq 1-\frac{6}{\sqrt{\pi}}  \frac{n^{3.5}d}{R-1}.  \]
\end{theorem}
\begin{proof}
For $\delta = C\sqrt{d}\epsilon$ with a parameter $C > 0$ to be specified later, we group the eigenvalue vectors $\Lambda_1,\ldots,\Lambda_n$ and its corresponding common eigenvectors $x_1,\ldots, x_n$  of $\mathcal A$ in $m$ clusters according to~\eqref{eq:clustercond}. We let  $\mathcal{X}_i = \Span\{x_j:K(j) = i\}$ denote the common invariant subspace spanned by the common eigenvectors belonging to the $i$th cluster.

Note that, in turn,  $\mathcal{X}_i$ is also an invariant subspace of $A(\mu)$. 
For the moment, we assume that
\begin{equation}\label{eq:eigengap_assum_nogap}
    \fnorm{E(\mu)} < |\langle\Lambda_i - \Lambda_j, \mu\rangle| / 4, \quad \text{ if } K(i) \neq K(j),
\end{equation}
where $E(\mu) := \mu_1E_1 + \cdots + \mu_d E_d$.
Along the lines of the proof of Theorem~\ref{nonasym theorem}, this allows us to apply Lemma~\ref{lemma:invsubspace} and conclude that
\begin{equation} \label{eq:splitoffdiag}
        \big\| \offdiag(\tilde{Q}^T{A}_k\tilde{Q}) \big\|_F^2 \le 
        \big\| ({A}_k - \bar \lambda^{(k)}_1 I ) \tilde X_1 \big\|_F^2 + \cdots + \big\| ({A}_k - \bar \lambda^{(k)}_m I ) \tilde X_m \big\|_F^2,
\end{equation}
where $\tilde X_i$ is an orthonormal basis of the perturbed invariant subspace $\tilde{\mathcal{X}}_i$ of $\tilde A(\mu)$ corresponding to $\mathcal{X}_i$. Each scalar
$\bar \lambda^{(k)}_i$ is chosen as the $k$th component of an arbitrary eigenvalue vector $\Lambda_j$ contained in the $i$th cluster, that is, $K(j) = i$.

We now analyze the first term in~\eqref{eq:splitoffdiag}. Without loss of generality, we may assume that $\mathcal{X}_1 = \Span  \{e_1,\ldots,e_{n_1}\}$ and thus each $A_k$ is block diagonal:
\[  A_k =  \begin{bmatrix}
A^{(k)}_{11}& 0 \\ 0 & A^{(k)}_{22}
\end{bmatrix}, \quad A^{(k)}_{11} \in \R^{n_1\times n_1}.  \]
Denoting $A_{ii}(\mu) = \mu_1 A^{(k)}_{ii} + \cdots +\mu_d  A^{(d)}_{ii}$, we define
\[\mathbb{T}_1(\mu): \reals^{(n-n_1) \times n_1} \to \reals^{(n-n_1) \times n_1}, \quad  \mathbb{T}_1(\mu) (X) = X A_{11}(\mu) - A_{22}(\mu)X.\]
Lemma \ref{lemma:invsubspace} applied to $\mathcal{X}_1$ implies that 
\begin{equation}
\label{eq:residual_nogap}
    ( {A}_k - \bar \lambda^{(k)}_1 I ) \tilde X_1 = \begin{bmatrix} (A^{(k)}_{11} - \bar \lambda_1^{(k)} I ) \hat X_1  \\ (A^{(k)}_{22} -\bar \lambda_1^{(k)}I) \mathbb{T}_1^{-1}(\mu)(B_1(\mu))C_1\end{bmatrix}, 
\end{equation}
where $\begin{bmatrix} \hat{X}_1 & 0\end{bmatrix}^T$ is a basis of an invariant subspace $\mathcal{X}_1$ with $\|\hat X_1\|_2 \leq 1$, and $B_1(\mu)$, $C_1$ are matrices satisfying

\[\fnorm{B_1(\mu)} \leq \sqrt{5} \fnorm{E(\mu)X_1}, \quad \eunorm{C_1} \leq 1,\quad \|\mathbb{T}_1^{-1}(\mu)(B_1(\mu))\| \leq 1/2.\]

Using that  $\|A^{(k)}_{11} - \bar \lambda_1^{(k)} I\|_2 \leq n_1\delta$ from~\eqref{eq:clustercond} and $\|\hat X_1\|_F \le \sqrt{n_1} \|\hat X_1\|_2 \le \sqrt{n_1}$, the first entry in~\eqref{eq:residual_nogap} satisfies
\begin{equation}
\label{eq:error_within_subspace}
\|(A^{(k)}_{11} - \bar \lambda_1^{(k)} I ) \hat X_1\|_F \le n_1^{3/2} \delta.
\end{equation}
To process the second entry in~\eqref{eq:residual_nogap}, we use the decomposition
\begin{align}
        &(A^{(k)}_{22} - \bar \lambda_1^{(k)}I) \mathbb{T}_1^{-1}(\mu)(B_1(\mu))C_1 \nonumber  \\  = & -\mathbb{T}^{(k)}_1 \compose \mathbb{T}_1^{-1}(\mu)(B_1(\mu))C_1 + \mathbb{T}_1^{-1}(\mu)(B_1(\mu)) (A^{(k)}_{11} - \bar\lambda_1^{(k)}I)C_1,
         \label{eq:error_perp_split}
\end{align}
where the linear matrix operator $\mathbb{T}_1^{(k)}$ is defined as $\mathbb{T}_1^{(k)}(X) = XA^{(k)}_{11} - A^{(k)}_{22}X$.

By diagonalizing $A_{11}^{(k)}, A_{22}^{(k)}$, it can be seen that the eigenvalues of the self-adjoint linear operator
\[\mathbb{T}^{(k)}_1 \circ \mathbb{T}_1^{-1}(\mu): \reals^{(n-n_1)\times n_1} \to \reals^{(n-n_1)\times n_1} \]
are given by
\[\frac{\lambda^{(k)}_i - \lambda_j^{(k)}}{\langle \Lambda_{i} - \Lambda_j, \mu \rangle}  \text{ for all } i,j \text{ such that } K(i) = 1 \text{ and } K(j) \neq 1.\]
Defining 
\[C_{\max}(\mu) :=
\max_{i,j,k \atop K(i) \not= K(j)}
\frac{|\lambda^{(k)}_i - \lambda_j^{(k)}|}{|\langle \Lambda_{i} - \Lambda_j, \mu \rangle|},\]
it thus follows that 
\[\big\|\mathbb{T}^{(k)}_1 \circ \mathbb{T}_1^{-1}(\mu)\big\| \leq C_{\max}(\mu),\]
where $\|\cdot\|$ denotes the norm induced by the Frobenius norm on $\reals^{(n-n_1) \times n_1}$. Note that this inequality also holds for the operators associated with the other summands in \eqref{eq:splitoffdiag}.

Thus, the two terms in~\eqref{eq:error_perp_split} are bounded by
\begin{align*}
\|\mathbb{T}^{(k)}_1 \compose \mathbb{T}_1^{-1}(\mu)(B_1(\mu))C_1\|_F  \le & \sqrt{5} C_{\max}(\mu) \|E(\mu)X_1\|_F \\
        \| \mathbb{T}_1^{-1}(\mu)(B_1(\mu)) (A^{(k)}_{11} - \bar\lambda_1^{(k)}I)C_1\|_F \le &
        \| \mathbb{T}_1^{-1}(\mu)(B_1(\mu))\|_F \|A^{(k)}_{11} - \bar\lambda_1^{(k)}I\|_2 \nonumber \\ 
        \le & n_1 \delta / 2,
\end{align*}
where we used Lemma \ref{lemma:invsubspace} and~\eqref{eq:clustercond} in the last inequality.  
Plugging these two inequalities together with~\eqref{eq:error_within_subspace} into~\eqref{eq:residual_nogap} gives
\[
\| ( {A}_k - \bar\lambda^{(k)}_1 I ) \tilde X_1 \|_F \le \sqrt{n_1^{3} \delta^2 + 5 C^2_{\max}(\mu) \|E(\mu)X_1\|^2_F} + n_1 \delta / 2.
\]
Analogous bounds, with $n_1$ replaced by $n_i$, hold for the other common invariant subspaces. Thus, we obtain from~\eqref{eq:splitoffdiag} that
\begin{align*}
\big\| \offdiag(\tilde{Q}^T{A}_k\tilde{Q}) \big\|_F & \le \Big( \sum_i \Big(
\sqrt{n_i^{3} \delta^2 + 5 C^2_{\max}(\mu) \|E(\mu)X_i\|^2_F} + \frac{1}{2} n_i \delta\Big)^2 \Big)^{1/2} \\
& \le  \Big( \sum_i n_i^{3} \delta^2 + 5 C^2_{\max}(\mu) \|E(\mu)X_i\|^2_F \Big)^{1/2} + \frac{1}{2}  \Big( \sum_i  n_i^2 \Big)^{1/2} \delta \\ 
& \le \big( n^3 \delta^2 + 5 C^2_{\max}(\mu) \|E(\mu)\|_F^2  \big)^{1/2} + n \delta / 2.
\end{align*}
Similarly as in~\eqref{eq:bound_total_error}, this yields
\[
    \big(\sum_{k} \fnormbig{\offdiag(\tilde{Q}^T\tilde{A}_k\tilde{Q})}^2\big)^{1/2} \leq 
    \big(1+ \sqrt{n^3C^2d^2 + 5dC^2_{\max}(\mu)\eunorm{\mu}^2} +  nCd / 2 \big)\epsilon,
\]
where we used $\delta = C\sqrt{d}\epsilon$ and $\fnorm{E(\mu)} \leq \eunorm{\mu}\epsilon$.
In other words, $R\epsilon$ is an upper bound on the off-diagonal norm as long as
\begin{equation} \label{eq:boundcmax_nogap}
    C_{\max}(\mu)\eunorm{\mu} \leq \sqrt{ (R-1 - nCd / 2)^2 - n^3C^2d^2 } / \sqrt{5d}.
\end{equation}

It remains to bound the probability that the inequalities \eqref{eq:eigengap_assum_nogap} or \eqref{eq:boundcmax_nogap} fail. Consider $(i,j)$ fixed with $K(i) \neq K(j)$.
Using the arguments from the proof of Theorem \ref{nonasym theorem} concerning~\eqref{eq:boundcmax}, the  probability that \eqref{eq:boundcmax_nogap} fails is bounded by
\[\Prob\Big( \Big|\Big\langle \frac{\Lambda_i - \Lambda_j}{\|\Lambda_i - \Lambda_j\|_2}, \frac{\mu}{\|\mu\|_2} \Big\rangle \Big| < \Big( \frac{5d}{{ (R-1 - nCd/2)^2 - n^3C^2d^2}} \Big)^{1/2} \Big).\]
Analogously, the probability that inequality \eqref{eq:eigengap_assum_nogap} fails satisfies
 \begin{align*}
 \Prob\Big( |\langle\Lambda_i - \Lambda_j, \mu\rangle| \leq 4\fnorm{E(\mu)} \Big) \nonumber 
 \le  \Prob\Big(\abs{\inner{\frac{\Lambda_i - \Lambda_j}{\eunorm{\Lambda_i - \Lambda_j}}}{\frac{\mu}{\eunorm{\mu}}}} < \frac{4}{C\sqrt{d}} \Big),
 \end{align*}
  where we used $\fnorm{E(\mu)} \leq \eunorm{\mu}\epsilon$ and~\eqref{eq:clustercond}.
Hence, the total probability that inequality~\eqref{eq:eigengap_assum_nogap} or~\eqref{eq:boundcmax_nogap} fails is bounded by 
\begin{equation}
\label{eq:failure_prob_no_gap}
    2\sqrt{\frac{2}{\pi}}\frac{n(n-1)}{C} + \sqrt{\frac{5}{2\pi}}\frac{dn(n-1)}{\sqrt{ (R-1 -  nCd / 2)^2 - n^3C^2d^2} }
\end{equation}
where we applied Lemma \ref{probability Bound} and a union bound for at most $n(n-1) /2$ different pairs of
 $(i,j)$ with $j > i$.
If we let
$R = 1 +  n C d / 2 + Cd\sqrt{n^3 + 5 / 16}$
or, equivalently,
\[C = \frac{R-1}{d\big( n / 2 + \sqrt{n^3 + 5 / 16}\big)},\]
the failure probability bound \eqref{eq:failure_prob_no_gap} becomes \[
    4\sqrt{\frac{2}{\pi}}\frac{n(n-1)}{C} = \sqrt{\frac{2}{\pi}}\frac{dn(n-1)(2n + \sqrt{16n^3 + 5})}{R-1}
    \leq \frac{6}{\sqrt{\pi}}  \frac{n^{3.5}d}{R-1},
\]
where the last inequality holds for all $n \geq 1$. This completes the proof.
\end{proof}


\subsection{Observed failure probability of RJD}
\label{subsec:failure_prob}

Let $\tilde{Q}$ denote the orthogonal matrix returned by Algorithm \ref{alg:rjd}. For given $R>1$, we say that Algorithm \ref{alg:rjd} fails when
\[\Big(\sum_{k=1}^{d}\fnormbig{\offdiag(\tilde{Q}^T\tilde{A}_k\tilde{Q})}^2\Big)^{1/2} \geq R\epsilon\]
for $(\fnorm{E_1}^2 + \cdots + \fnorm{E_d}^2)^{1/2} \le \epsilon$. From Theorems~\ref{nonasym theorem} and~\ref{theorem no gap}, we expect the failure probability of Algorithm \ref{alg:rjd} with $L$ trials to be proportional to $1/(R-1)^L$. To verify this experimentally, we consider $n= 10$, $d=5$ and chose an input error $\epsilon = 10^{-5}$, which dominates round-off error. We let $\tilde{\Acal}= \{A_k + E_k\}_{k=1}^d$ where $\Acal = \{A_k\}_{k=1}^d$ is an exactly commuting family, generated by random diagonal matrices transformed with a common random orthogonal matrix. The input error matrices $E_k$ are random symmetric matrices scaled such that $(\fnorm{E_1}^2 + \cdots + \fnorm{E_d}^2)^{1/2} =  \epsilon$. We repeat Algorithm~\ref{alg:rjd} $10^8$ times for different values of $L$ and use the failure frequency to approximate the failure probability for each $R$. As clearly seen from  Figure~\ref{probability_vs_tolerance}, the failure probabilities behave as predicted from Theorem \ref{nonasym theorem} and Theorem \ref{theorem no gap}; the dependence on $R$ is optimal up to constants.

\begin{figure}[h]
    \centering
    \includegraphics[width=1\textwidth]{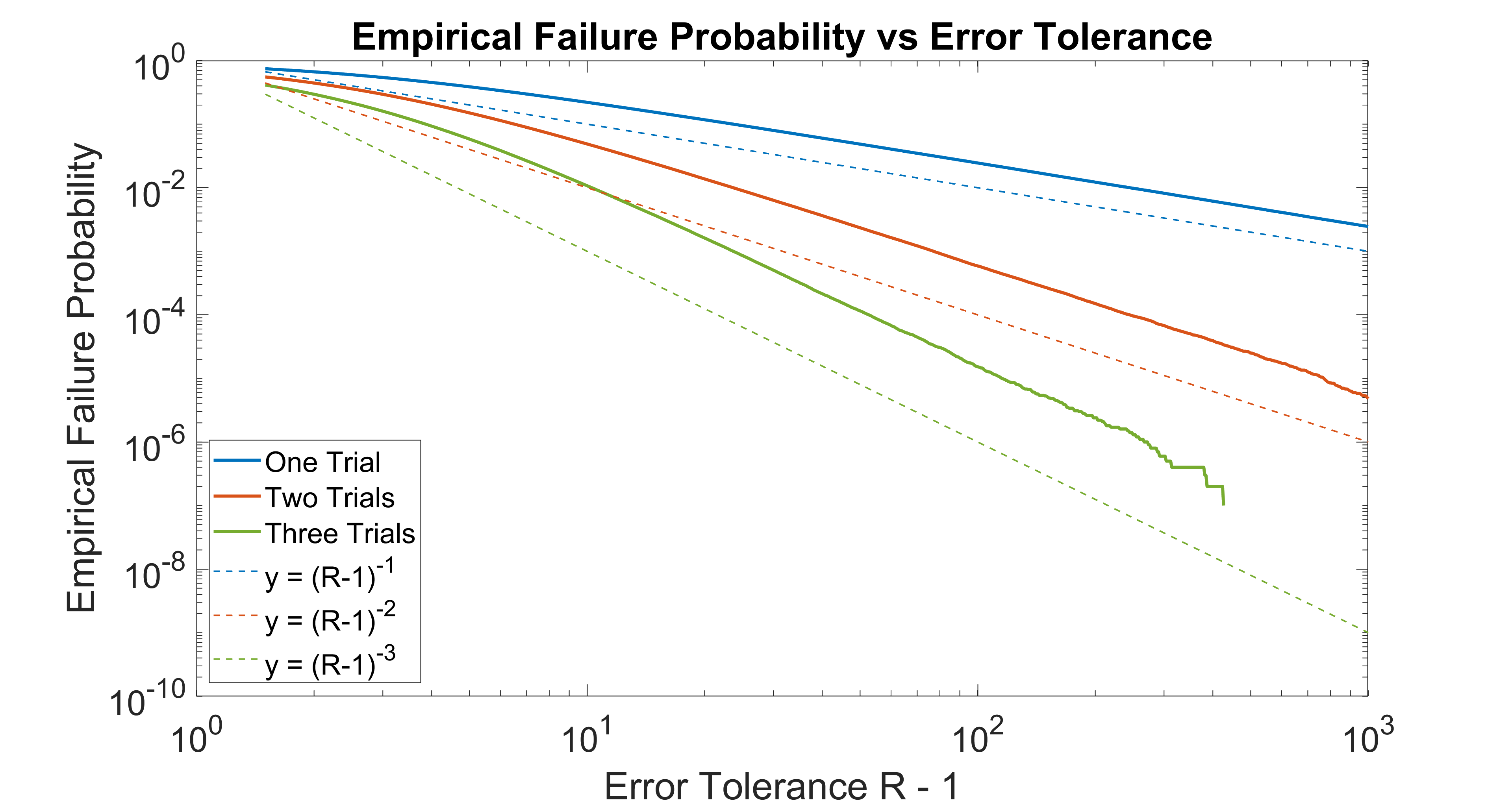}
    \caption{Empirical failure probability vs. $R - 1$ on log-log scale.}
    \label{probability_vs_tolerance}
\end{figure}

%% file: deflation_algorithm.tex
\section{A deflation-based RJD algorithm}
\label{sec:deflation}

Algorithm~\ref{alg:rjd} chooses the best orthogonal matrix among $L$ samples and discards the other $L-1$ samples. This is not necessarily the best use of the information contained in the samples. Specifically, the discarded samples could still contain good approximations of common eigenvectors for parts of the spectrum. To extract these approximations, we note that
\[
   \big\|\offdiag(\tilde{Q}^T\tilde{A}_k\tilde{Q})(:,i)\big\|_2 =
   \big\|\big( \tilde{A}_k - (\tilde{Q}(:,i)^T\tilde{A}_k\tilde{Q}(:,i)) I_n)\tilde{Q}(:,i)\big\|_2,
\]
for an orthogonal matrix $\tilde Q$, 
where Matlab's colon notation is used. In other words, the columns of the off-diagonal error contain the residuals for the eigenvector approximations of $A_k$. This allows us to identify and extract columns of $\tilde Q$ with a residual below a certain threshold. Although it is desirable to relate the threshold to the input error; the latter is usually unknown. Instead, we estimate the threshold from the best eigenvector error observed within $L$ trials. We can then partition columns of $\tilde Q$ into $ \tilde Q_{\text{suc}}$ and $\tilde Q_{\text{fail}}$ where $ \tilde Q_{\text{suc}}$ contains all the columns with residuals below the estimated threshold and $\tilde Q_{\text{fail}}$ contains the remaining. We deflate $ \tilde Q_{\text{suc}}$ by restricting each $\tilde A_k$ to $\tilde Q_{\text{fail}}$, i.e., $\tilde A_k = \tilde Q_{\text{fail}}^T \tilde A_k\tilde Q_{\text{fail}}$, and recursively solve the deflated subproblem.
These ideas lead to Algorithm~\ref{alg:drjd} (DRJD).

\begin{algorithm}[h]
\caption{\textbf{D}eflation-based \textbf{R}andomized \textbf{J}oint \textbf{D}iagonalization (DRJD)}
\textbf{Input: } \text{Family of symmetric matrices $\tilde{\Acal} = \{\tilde{A}_k \in \reals^{n \times n}\}_{k=1}^{d}$, number of trials $L$.}\\
 \textbf{Output: } \text{Joint approximate diagonalizer $\tilde{Q}$.}
\begin{algorithmic}
\label{alg:drjd}
\FOR{$i=1$ to $L$}
    \STATE
    Compute $\tilde{Q}_i$ using Algorithm \ref{alg:rjd} with $1$ trial.
    \STATE $t_i = \min_{j} \big\{ \sum_{k=1}^{d}\big\|\offdiag(\tilde{Q}_i^T\tilde{A}_k\tilde{Q}_i)(:,j)\big\|^2_2 \big \}$.
\ENDFOR
\STATE $t = 2 \min\{t_1, \dots, t_L\}$. \tcp{Scale by $2$ to avoid over estimation.}
\FOR{$i=1$ to $L$}
    \STATE $\ind_i = \big\{j: \sum_{k=1}^{d}\big\|\offdiag(\tilde{Q}_i^T\tilde{A}_k\tilde{Q}_i)(:,j)\big\|^2_2 \leq t \big\}$.
\ENDFOR
\STATE $i^* = \argmax_{i }\{ |\ind_i| \}$. \tcp{Maximize the number of successful columns.}
\STATE $\tilde Q_{\text{suc}} = \tilde Q_{i^*}(:,\ind_{i^*})$. \STATE $\tilde Q_{\text{fail}} = \tilde Q_{i^*}(:,\{1,\ldots,n\} \setminus \ind_{i^*})$.
\IF{ $\tilde Q_{\text{fail}}$ is empty}
 \RETURN $\tilde Q \equiv \tilde Q_{\text{suc}}$
\ELSE
\STATE Recursively compute $\tilde{Q}_\text{rec}$ = DRJD$(\{ \tilde Q_{\text{fail}}^T \tilde{A}_k \tilde Q_{\text{fail}} \}_{k=1}^{d}, L)$.
\RETURN $\tilde Q \equiv \begin{bmatrix}
\tilde Q_{\text{suc}} & \tilde Q_{\text{fail}} \tilde{Q}_\text{rec}\end{bmatrix}$
\ENDIF
\end{algorithmic}
\end{algorithm}

%% file: numerical_experiments.tex
\section{Numerical experiments}
\label{sec:numeric}
In this section, we compare RJD and DRJD with existing state-of-the-art algorithms on joint diagonalization of real symmetric matrices on synthetic data and real applications, including Blind Source Separation and Single Topic Models. All algorithms and numerical experiments in this section are implemented in Python 3.8 and executed on a Dell XPS 13 2-In-1
with Intel Core i7-1165G7 CPU and 16 GB of RAM. The implementation is available at \url{https://github.com/haoze12345/symmetric_rjd}. 

Throughout the experiments, we only consider $L=3$ trials to boost the success probability for RJD and DRJD. Obviously, the accuracy of RJD and DRJD can always be improved further by allowing for more trials, at the expense of running time. For all optimization-based algorithms, we use the identity matrix as the initial value.

\subsection{Synthetic data}
In this experiment, RJD and DRJD are compared to existing optimization-based JD methods with orthogonal joint diagonalizers, including JADE \cite{MR1372928}, FFDIAG \cite{MR2247999}, PHAM \cite{Pham2001} and a recently developed Quasi-Newton Method QNDIAG \cite{ablin:hal-01936887} for synthetic data. 

The nearly commuting matrices  $\Acal = \{A_k + E_k \in \reals^{n \times n}\}_{k=1}^{d}$ are generated for different values of $n,d$ in the way described in Section \ref{subsec:failure_prob}. As QNDIAG and PHAM assume positive definite matrices, we enforce  positive definiteness by picking the entries of the diagonal matrices from a uniform distribution on the interval $[0.01, 1.01]$. Note that PHAM does not ensure the output joint diagonalizer to be orthogonal. 

We consider three noise levels $\epsilon = (\sum_{k=1}^{d}\fnorm{E_k}^2)^{1/2}$, $\epsilon_1 = 0$, $\epsilon_2 = 10 ^{-5}$ and $\epsilon_3 = 10^{-1}$. For each setting of $n,d$ and noise level, we repeat the experiment $100$ times on the same input matrices and report the average running time and error. The comparisons are shown in Tables \ref{table:synthetic1}--\ref{table:synthetic3}. All running times are reported in milliseconds and the reported error is the square root of
$\mathcal L(\tilde Q)$ from~\eqref{eq:least_square_measure}.
\input{table_error_runtime_newdeflation}
The tables clearly show the advantages of RJD and DRJD. When roundoff error is the only noise introduced in the matrices, both algorithms are much faster than any of the other algorithms, while returning a similar level of output error. As the noise level increases, DRJD continues to achieve comparably good accuracy, while the error of RJD is sometimes considerably larger. DRJD is also significantly faster than most optimization-based algorithms, with the notable exception of QNDIAG. Note, however, that QNDIAG is restricted to positive definite matrices. 

\subsection{Blind source separation}

First, let us briefly recall Blind Source Separation (BSS) with instantaneous mixture. Consider $m$ source signals $s_j(t)$, $j = 1,...m$, and $n$ observed signals $x_i(t)$, $i = \{1,\dots,n\}$. We assume that the observed signals are instantaneous mixture of the source signals as follows:
\[x_i(t) = \sum_{j = 1}^{m} A_{ij}s_j(t).\]
We assume that $m=n$ and that the mixing matrix $A$ is non-singular. For the source signals, we assume that $s_1(t), \ldots, s_n(t)$ are jointly stationary random processes, that there is at most one Gaussian source and that for each $t$, the signals $s_1(t),\ldots, s_n(t)$ are mutually independent random variables. 

The task of BSS is to find an unmixing matrix $B$ such that each $(Bx)_j(t)$ is proportional to some source $s_i(t)$. If the true mixing matrix $A$ is known, the performance of the obtained unmixing matrix $B$ can be measured with the Moreau-Amari~\cite{NIPS1995_e19347e1} (MA)  index defined as follows:
\[I_{\text{MA}}(M) = \frac{1}{2n(n-1)} \sum_{i=1}^{n}\Big(\frac{\sum_{j=1}^n\
\abs{M_{ij}}}{\max_j{\abs{M_{ij}}}} + \frac{\sum_{j=1}^n\
\abs{M_{ji}}}{\max_j{\abs{M_{ji}}}} - 2\Big), \]
where $M = BA$. Notice that $I_{\text{MA}}(M) \in [0,1]$ and $I_{\text{MA}}(M) = 0$ when $B=A^{-1}$.

Given fixed $M \in \reals^{n \times n}$ and (random) signals $x_i$, the \emph{cumulant matrix} $Q_x(M) \in \reals^{n \times n}$ is defined in~\cite{10.1162/089976699300016863} as follows:
\[Q_x(M)_{ij} := \sum_{k,l=1}^{n}\cum(x_i,x_j,x_k,x_l)M_{kl},\]
where $\cum$ denotes the joint cumulant of random variables.

It is shown that after pre-whitening, 
$\{Q_x(M): M \in \reals^{n \times n}\}$ is an exactly commuting family, and by performing JD, we can recover the unmixing matrix $B$. In this application, the level of the input error is often small and the size of the involved matrices is not large.  To test RJD and DRJD, we perform BSS by JD on the same three audio sources described in~\cite{MiettinenJari2017Bssb} mixed with another white noise signal with standard deviation $\sigma = 0.01$ by a random orthogonal matrix.
In our experiment, there are $d=10$ $4 \times 4$ matrices to be jointly diagonalized. The original signals, mixed signals, unmixed signals by RJD and unmixed signals by DRJD  are shown in Figure \ref{fig:original_signal}, Figure \ref{fig:mixed_signal}, Figure \ref{fig:unmixed_signal_RJD} and Figure \ref{fig:unmixed_signal_DRJD} respectively. Visually,  BSS with RJD and DRJD recovers the original signals accurately. 
 \begin{figure}[ht]
    \centering
    \includegraphics[width=1\textwidth]{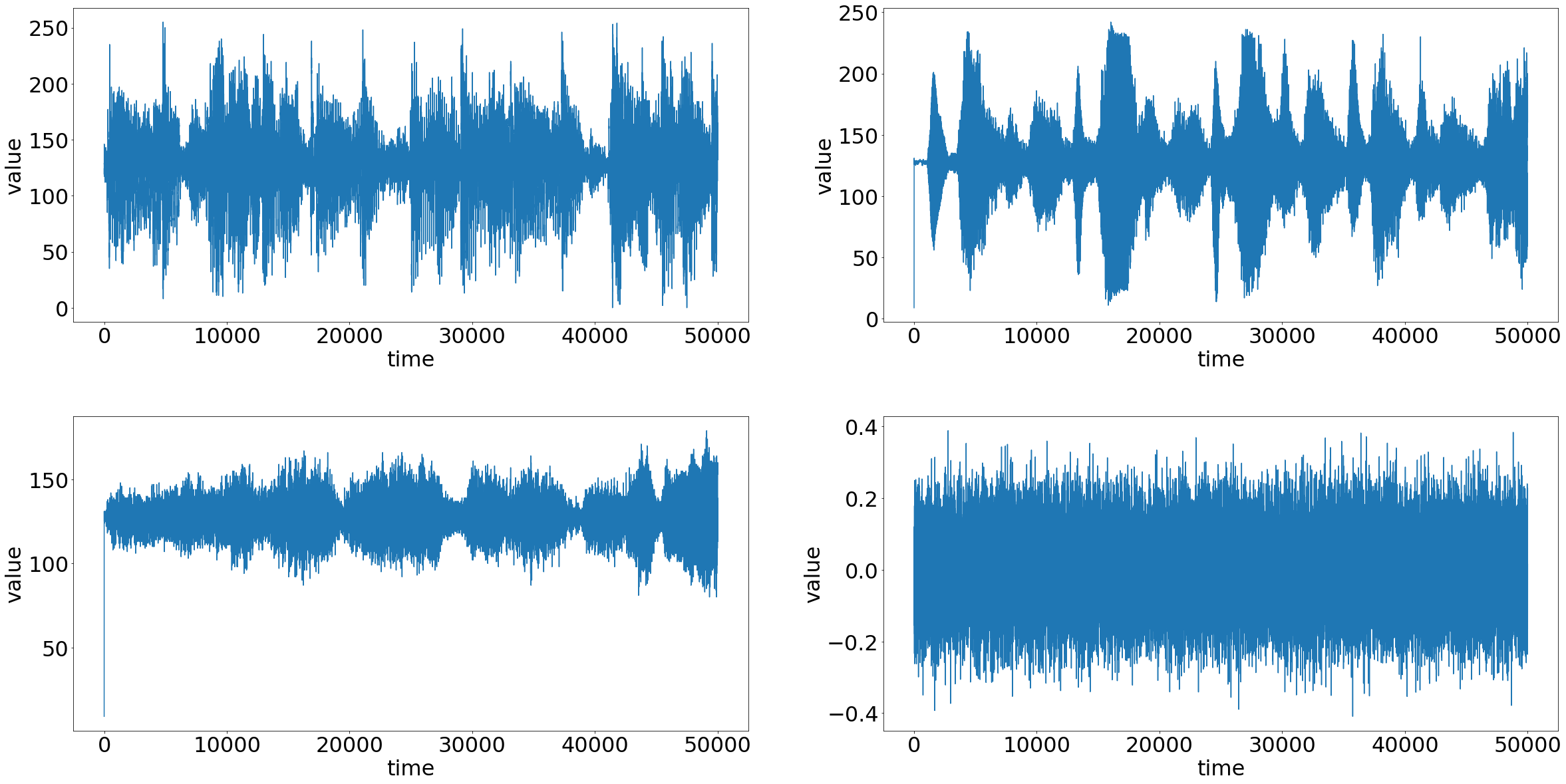}
    \caption{Original signals before mixture, the first three signals are the audio sources and the rest is white noise. All signals are sorted according to the energy.}
    \label{fig:original_signal}
\end{figure}
\begin{figure}[ht]
    \centering
    \includegraphics[width=1\textwidth]{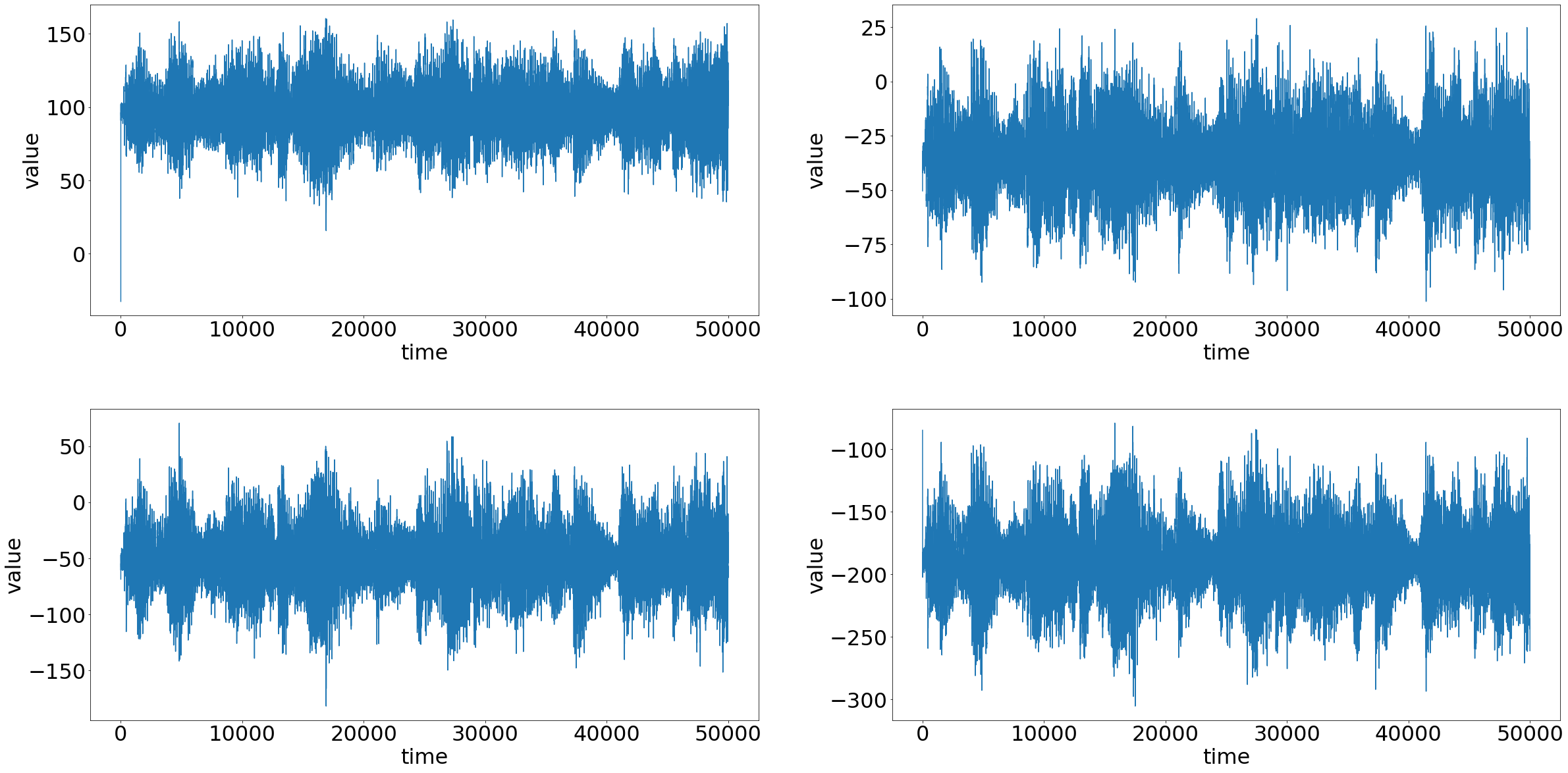}
    \caption{Signals after mixture}
    \label{fig:mixed_signal}
\end{figure}
\begin{figure}[ht]
    \centering
    \includegraphics[width=1\textwidth]{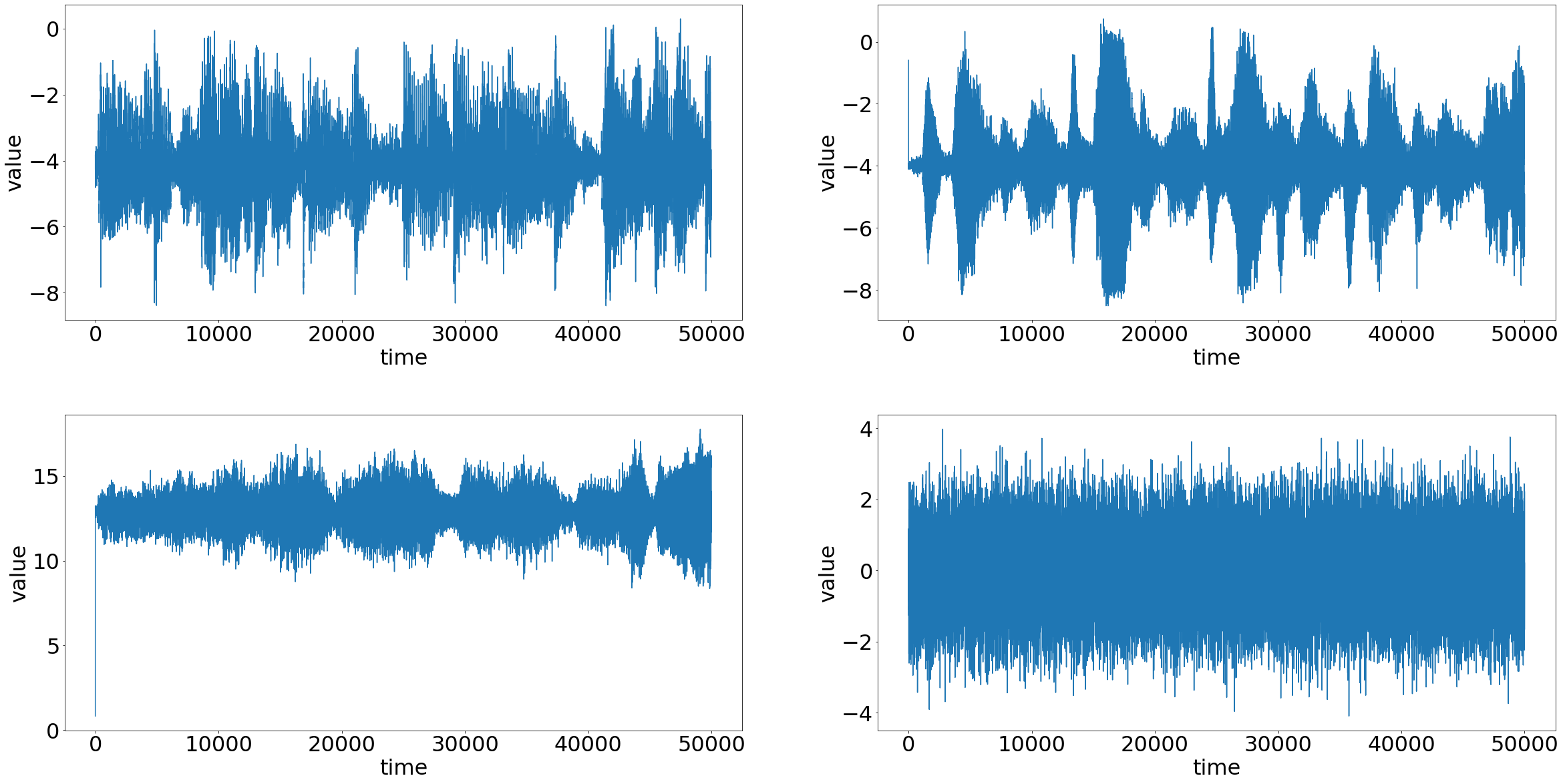}
    \caption{Signals recovered by RJD, signals are sorted according to the energy. The first three signals correspond to the original audio sources before mixture in Figure~\ref{fig:original_signal}.}
    \label{fig:unmixed_signal_RJD}
\end{figure}
\begin{figure}[ht]
    \centering
    \includegraphics[width=1\textwidth]{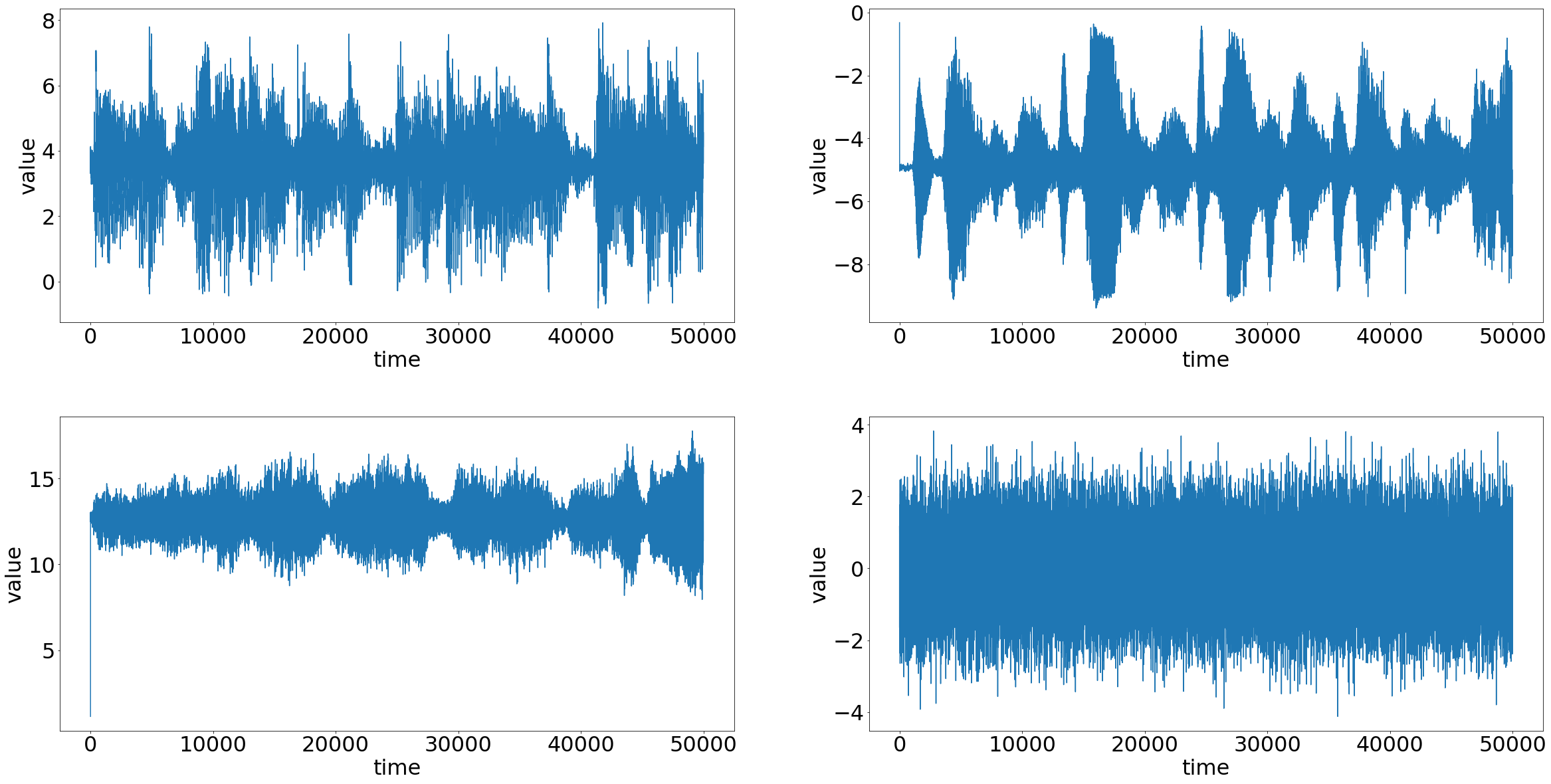}
    \caption{Signals Recovered by DRJD, signals are sorted according to the energy. The first three signals correspond to the original audio sources before mixture in Figure \ref{fig:original_signal}.}
    \label{fig:unmixed_signal_DRJD}
\end{figure}

We also compare RJD and DRJD with JADE and FFDIAG for the same audio data. Note that QNDIAG and PHAM are not applicable because the involved matrices are not positive definite. We perform the experiments 100 times with the same $4$ signals and the same mixing matrix $A$ and record the average MA indices and running time in Table \ref{table:MAindex}. Both RJD and DRJD are significantly faster than the other algorithms and, additionally, DRJD obtains significantly better separation.
 \input{BSS_runtime_loss}

To further demonstrate the advantages of our randomized algorithm,  we consider JD on Fourier cospectra of the electroencephalogram (EEG) recordings; see~\cite{CONGEDO20082677}. In this application, the input error tends to be not small and we therefore exclude RJD from the comparison. The Fourier cospectra are computed on $84$ resting-state of healthy individuals with pre-whitening and dimension reduction to keep $99.9\%$ of the total variance. For each recording of EEG, we obtain $47$ matrices of size roughly $20 \times 20$ to be jointly diagonalized\interfootnotelinepenalty=10000\footnote{The data is provided by Marco Congedo and is available at \url{https://github.com/Marco-Congedo/STUDIES/tree/master/AJD-Algos-Benchmark}.}.
As the involved matrices are positive definite, we compare DRJD with JADE, QNDIAG, PHAM, and FFDIAG.
Since no ground truth is known in this experiment, we only compare them in terms of running time. For each recording, we jointly diagonalize its Fourier cospectra with different JD algorithms $100$ times and report the running time; see Figure~\ref{fig:egg}. 
It can be clearly seen that  DRJD is at least one order of magnitude faster than the other algorithms.
\begin{figure}[ht]
    \centering
    \includegraphics[width=1\textwidth]{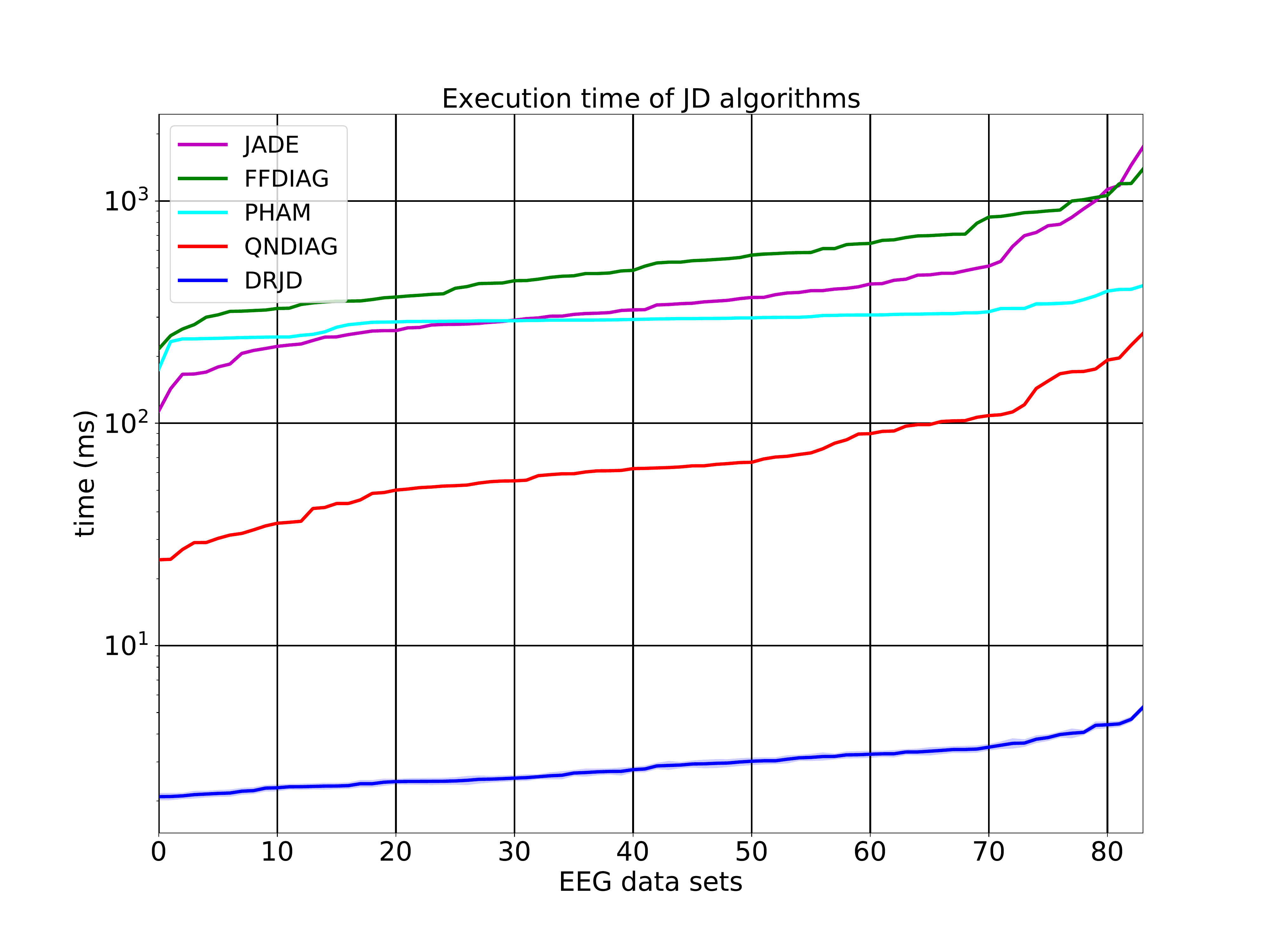}
    \caption{EEG JD running time. The x-axis is the index of the recording, and the y-axis is the running time in log scale. The curves correspond to the averaged running time over $100$ repeats for each recording of different algorithms, and the shaded areas correspond to the $95$\% confidence interval of the running time.}
    \label{fig:egg}
\end{figure}


\subsection{Single topic models}
The JD problem also arises in learning Latent Variable Models, and the Single Topic Model is a special case of Latent Variable Models. It is a simplified model to describe how documents with different topics are generated. Given $N$ documents and $k$ topics,  each document is assumed to have only one topic, and all topics share a common vocabulary of size $n$. Then the model is determined by two parameters, $\omega \in \reals^k$, the probability vector of the topics and $M \in \reals^{n \times k}$, the probability matrix of words given different topics. The documents are generated as follows:
\begin{itemize}
    \item First, a (latent) topic $H \in \{1,\ldots,k\}$ is picked according to the following discrete distribution:
    \[\omega_i := P\big(H=i\big) \text{ with } \omega = \big(\omega_1, \dots , \omega_k\big)^T .\]
    \item Once the topic of the document is fixed, the $t$th word $W_t$ of the document is drawn from the following distribution:
    \[P\big( \text{$W_t$ is the $j$-th word in the vocabulary}| H = i\big) = M_{ji} ,\]
    and we use $\mu_i \in \reals^{n}$ to denote the $i$th column of $M$, which corresponds to the $i$th topic.
\end{itemize}
Let $x_t$ denote the random variable that $x_t = e_i \in \reals^n$ where $e_i$ is the $i$th standard  basis vector if and only if the $t$th word in the document is the $i$th word in the vocabulary.
Then it is shown in \cite[Theorem 3.1]{MR3332930} that 
\[M_2 :=\Exp[x_1 \tensor x_2] = \sum_{i=1}^{k}\omega_i \mu_i \tensor \mu_i,\]
\[M_3 :=\Exp[x_1 \tensor x_2 \tensor x_3] = \sum_{i=1}^{k}\omega_i \mu_i \tensor \mu_i \tensor \mu_i.\]
It is also shown in \cite{MR3332930} that after whitening, the orthogonal joint diagonalization of $M_2$ and slices of $M_3$ can recover the parameters $\omega$ and $\mu_1, \ldots, \mu_k$.

To demonstrate the feasibility of our DRJD for Single Topic Models,  we apply our algorithm to the ``20Newsgroup" dataset \footnote{\url{https://scikit-learn.org/0.19/datasets/twenty_newsgroups.html}}, where each document has only one ground truth topic, and topics are not closely related to each other. We pick documents from the topics `sci.med', `sci.space', `talk.politics.guns' and `alt.atheism'. Then, we pre-process the data by filtering out the stop words and words with very high frequency ($\geq 80\%$). Next, we only include the top $1000$ most frequent words into the vocabulary. Afterwards, we apply the tensor retrieval method described in \cite{RuffiniMatteo2018Anmo} to construct $M_2$ and $M_3$ from the data. Finally, we apply our DRJD to estimate the parameters $\omega$ and $\mu_1, \ldots, \mu_k$. We sort each topic $i$ according to its probability $\omega_i$, and we show the top 10 words according to the probability vector $\mu_i$. The result is shown in Table \ref{table:20news_table_words}. For a reference, we also show the result from the recently developed Singular Value Based Tensor Decomposition (SVTD)\cite{RuffiniMatteo2018Anmo}, an algorithm designed primarily for learning Latent Variable Models. From the tables, we can see that the content of each topic can be easily inferred from the top 10 words for both algorithms, and topics are sorted in the same order, demonstrating that both algorithms are suitable for this dataset.
\begin{table}
\caption{\label{table:20news_table_words} 20NewsGroup - Top 10 words for each topic}
\begin{center}
    \begin{tabular}{cc} DRJD & SVTD \\
    \input{table_20news_drjd} & \input{table_20news_svtd}
    \end{tabular}
\end{center}
\end{table}

Then we compare DRJD with other Latent Variable Model algorithms, including SVTD and the well-known Tensor Power Method (TPM) \cite{MR3332930}. The quantitative measure of performance we use here is the topic \textit{coherence} introduced in \cite{DBLP:conf/emnlp/MimnoWTLM11}. For a topic with word probability distribution $\mu$, its coherence is defined as
\[\Cohe(\mu) = \sum_{j=2}^{L} \sum_{i=1}^{j-1} \log\frac{D(w_i,w_j)+1}{D(w_i)}\]
where $(w_1,\dots,w_L)$ is the list of top $L = 15$ most popular words in the topic $\mu$, $D(w_i,w_j)$ is the count of documents having word $w_i$ and $w_j$ and $D(w_i)$ is defined analogously. The higher the score, the better the coherence is. As in \cite{RuffiniMatteo2018Anmo}, we compute the mean coherence score for the topics. For these three algorithms, we again report their running time and mean topic coherence score for the same data averaged over 100 runs. The comparison is shown in Table \ref{table:20news_coherence}.
\input{table_20news_coherence}

It is demonstrated that our DRJD is a suitable candidate for learning the Single Topic Model. It can reveal the latent topics successfully. Also, it can achieve a slightly better topic coherence than the popular TPM algorithm with a drastically faster speed ($50 \times$ faster).  Compared to SVTD, which is primarily designed for this application, our DRJD still runs significantly faster ($10 \times$ faster), but SVTD achieves the best topic coherence among the three algorithms.

%% file: table_error_runtime_newdeflation.tex
\begin{table}[!hbt!]
\begin{center}
\caption{Running time and accuracy comparison for $d=10, n=10$}
\label{table:synthetic1}
\small
\begin{tabular}{||c|c|c|c|c|c|c||}
\hline
Name & Time $\epsilon_1$ & Error $\epsilon_1$ & Time $\epsilon_2$ & Error $\epsilon_2$ &Time $\epsilon_3$ &Error $\epsilon_3$\\ 
\hline
JADE & $\num{11.2}$ & $\num{2.0e-08}$ & $\num{10}$ & $\num{8.1e-06}$ & $\num{12.3}$ & $\num{8.0e-02}$\\

FFDIAG & $\num{13.3}$ & $\num{1.9e-08}$ & $\num{12}$ & $\num{8.1e-06}$ & $\num{11.7}$ & $\num{8.0e-02}$\\

PHAM & $\num{29.1}$ & $\num{4.9e-13}$ & $\num{26.1}$ & $\num{2.7e-05}$ & $\num{28.2}$ & $\num{2.2e-01}$\\

QNDIAG & $\num{4.14}$ & $\num{5.5e-14}$ & $\num{4.36}$ & $\num{8.5e-06}$ & $\num{5.33}$ & $\num{8.5e-02}$\\

RJD & $\num{0.389}$ & $\num{2.5e-14}$ & $\num{0.518}$ & $\num{2.0e-05}$ & $\num{0.467}$ & $\num{2.0e-01}$\\

DRJD & $\num{0.448}$ & $\num{2.5e-14}$ & $\num{2.21}$ & $\num{1.1e-05}$ & $\num{1.96}$ & $\num{1.1e-01}$\\

\hline
\end{tabular}
\end{center}
\end{table}

\begin{table}[!hbt!]
\begin{center}
\caption{Running time and accuracy comparison for $d=10, n=100$}
\label{table:synthetic2}
\small
\begin{tabular}{||c|c|c|c|c|c|c||}
\hline
Name & Time $\epsilon_1$ & Error $\epsilon_1$ & Time $\epsilon_2$ & Error $\epsilon_2$ &Time $\epsilon_3$ &Error $\epsilon_3$\\ 
\hline
JADE & $\num{1.93e+03}$ & $\num{2.7e-07}$ & $\num{1.92e+03}$ & $\num{9.3e-06}$ & $\num{2.12e+03}$ & $\num{9.3e-2}$\\

FFDIAG & $\num{3.03e+03}$ & $\num{1.9e-09}$ & $\num{3.09e+03}$ & $\num{9.3e-06}$ & $\num{3.06e+03}$ & $\num{9.3e-2}$\\

PHAM & $\num{4.43e+03}$ & $\num{5.2e-10}$ & $\num{4.52e+03}$ & $\num{1.1e-4}$ & $\num{4.51e+03}$ & $\num{1.1}$\\

QNDIAG & $\num{4.78e+2}$ & $\num{8.7e-14}$ & $\num{3.81e+2}$ & $\num{1.0e-05}$ & $\num{2.24e+2}$ & $\num{0.1}$\\

RJD & $\num{1.73e+1}$ & $\num{8.7e-12}$ & $\num{1.67e+1}$ & $\num{4.9e-4}$ & $\num{1.68e+1}$ & $\num{2.0}$\\

DRJD & $\num{1.80e+1}$ & $\num{1.8e-10}$ & $\num{3.67e+2}$ & $\num{1.3e-05}$ & $\num{3.53e+2}$ & $\num{0.13}$\\

\hline
\end{tabular}
\end{center}
\end{table}

\begin{table}[!hbt!]
\begin{center}
\caption{Running time and accuracy comparison for $d=30, n=30$}
\label{table:synthetic3}
\small
\begin{tabular}{||c|c|c|c|c|c|c||}
\hline
Name & Time $\epsilon_1$ & Error $\epsilon_1$ & Time $\epsilon_2$ & Error $\epsilon_2$ &Time $\epsilon_3$ &Error $\epsilon_3$\\ 
\hline
JADE & $\num{153}$ & $\num{1.4e-07}$ & $\num{274}$ & $\num{9.5e-06}$ & $\num{289}$ & $\num{9.5e-02}$\\

FFDIAG & $\num{241}$ & $\num{1.3e-10}$ & $\num{248}$ & $\num{9.5e-06}$ & $\num{248}$ & $\num{9.5e-02}$\\

PHAM & $\num{350}$ & $\num{5.5e-14}$ & $\num{468}$ & $\num{3.5e-05}$ & $\num{462}$ & $\num{3.3e-01}$\\

QNDIAG & $\num{32.8}$ & $\num{2.0e-14}$ & $\num{8.59}$ & $\num{9.8e-06}$ & $\num{9.91}$ & $\num{9.8e-02}$\\

RJD & $\num{2.36}$ & $\num{3.9e-12}$ & $\num{2.16}$ & $\num{1.6 e-4}$ & $\num{1.97}$ & $\num{1.15e+00}$\\

DRJD & $\num{2.33}$ & $\num{4.4e-12}$ & $\num{9.70}$ & $\num{1.4e-05}$ & $\num{9.98}$ & $\num{1.4e-01}$\\

\hline
\end{tabular}
\end{center}
\end{table}

%% file: BSS_runtime_loss.tex
\begin{table}[!hbt!]
\begin{center}
\caption{Running time and MA index comparison for audio data}
\label{table:MAindex}
\begin{tabular}{||c | c | c | |}
\hline
Name & Avg time(ms) & Avg MA index\\
\hline
FFDIAG & $\num{4.645922}$ & $\num{0.074119}$\\
\hline
JADE & $\num{2.782023}$ & $\num{0.074109}$\\
\hline
RJD & $\num{0.299520}$ & $\num{0.074526}$\\
\hline
DRJD & $\num{1.228833}$ & $\num{0.064137}$\\
\hline
\end{tabular}
\end{center}
\end{table}

%% file: table_20news_DRJD.tex
{\scriptsize \begin{tabular}{|c|c|c|c|c|c|}
\hline
Topic 1 & Topic 2 & Topic 3 & Topic 4\\
\hline
god & space & health & file\\
jesus & launch & 1993 & gun\\
people & nasa & hiv & congress\\
atheists & satellite & use & control\\
atheism & edu & medical & firearms\\
does & data & 10 & mr\\
matthew & commercial & number & states\\
religious & satellites & 20 & united\\
just & year & aids & rkba\\
believe & market & april & house\\
\hline
\end{tabular}}

%% file: table_20news_svtd.tex
{\scriptsize \begin{tabular}{|c|c|c|c|c|c|}
\hline
Topic 1 & Topic 2 & Topic 3 & Topic 4\\
\hline
jesus & space & health & file\\
god & launch & hiv & gun\\
atheists & satellite & 1993 & congress\\
atheism & commercial & use & control\\
people & market & medical & firearms\\
matthew & satellites & 10 & mr\\
religious & data & aids & states\\
religion & year & number & united\\
does & nasa & 20 & rkba\\
believe & new & april & house\\
\hline
\end{tabular}}

%% file: table_20news_coherence.tex
\begin{table}[!hbt!]
\begin{center}
\caption{Running time and coherence comparison for 20NewsGroup}
\label{table:20news_coherence}
\begin{tabular}{||c | c | c | |}
\hline
Name & Avg time(ms) & Avg coherence\\
\hline
SVTD & $\num{45.869329}$ & $\num{-193.988781}$\\
\hline
TPM & $\num{298.112507}$ & $\num{-198.812939}$\\
\hline
DRJD & $\num{3.998115}$ & $\num{-196.973950}$\\
\hline
\end{tabular}
\end{center}
\end{table}

%% file: conclusion.tex
\section{Conclusion}
\label{sec:conclusion}

In this paper, we proposed two randomized algorithms (RJD and DRJD) to jointly (and approximately) diagonalize a family of real symmetric matrices. Our main result shows that RJD returns, with high probability, an orthogonal transformation with an off-diagonal error on the level of the input error. Numerous numerical experiments show RJD is exceptionally efficient for matrices that are very nearly commuting, while DRJD achieves a good balance between accuracy and running time compared to other state-of-art optimization-based algorithms on both synthetic and real-world data. The advantages of RJD and DRJD become particularly relevant for a large family of matrices.

\begin{paragraph}{Acknowledgments.} The authors thank the referees for helpful remarks, which improved the presentation of this manuscript. The second author thanks Stefan Kunis for a discussion, related to~\cite{MR3936543}, which inspired this work.
\end{paragraph}